\documentclass[12pt]{article}

\usepackage[T1]{fontenc}
\usepackage{lmodern}

\usepackage{amsmath,amssymb,amsthm}
\usepackage{mathtools}
\usepackage{a4wide}
\usepackage{hyperref}
\hypersetup{
    colorlinks=true,
    citecolor=blue,
    filecolor=black,
    linkcolor=blue,
    urlcolor=blue,
    pdftitle={Classification of Hom-Lie structures on finite-dimensional simple Lie algebras},
    pdfauthor={Jiacheng Sun, Chi Zhang, and Haoran Zhu},
    pdfkeywords={Hom-Lie structure, simple Lie algebra, Zassenhaus algebra, Cartan type, Melikian algebra}
}

\DeclareMathOperator{\Hom}{Hom}
\DeclareMathOperator{\End}{End}
\DeclareMathOperator{\gr}{gr}
\DeclareMathOperator{\Span}{span}
\DeclareMathOperator{\diver}{div}
\DeclareMathOperator{\Cent}{C}
\DeclareMathOperator{\HomLie}{HomLie}

\newcommand{\K}{\mathbb K}
\newcommand{\id}{\operatorname{id}}

\newcommand{\OO}{\mathcal O}
\newcommand{\uln}{\underline n}

\newtheorem{theorem}{Theorem}[section]
\newtheorem{lemma}[theorem]{Lemma}
\newtheorem{proposition}[theorem]{Proposition}
\newtheorem{corollary}[theorem]{Corollary}
\theoremstyle{remark}
\newtheorem{remark}[theorem]{Remark}

\numberwithin{equation}{section}

\title{\Large\bfseries Classification of Hom--Lie structures on
finite-dimensional simple Lie algebras}
\author{Jiacheng Sun\qquad Chi Zhang\qquad Haoran Zhu}
\date{}

\begin{document}

\maketitle

\begin{abstract}
We classify the Hom--Lie structures on finite-dimensional simple Lie algebras
over an algebraically closed field of characteristic different from \(2\) and
\(3\).  Except for the three-dimensional simple Lie algebra and the Zassenhaus
algebras, every such structure is a scalar multiple of the identity.  In the
two exceptional cases, we determine the full Hom--Lie space: it has dimension
six for the three-dimensional algebra, while for a Zassenhaus algebra it
consists precisely of the multiplication operators arising from the
corresponding divided-power algebra.  This proves a conjecture of Makhlouf and
Zusmanovich.
\end{abstract}



\vspace{5mm}

\section{Introduction}

Let \(L\) be a Lie algebra over a field \(\K\).  A \emph{Hom--Lie structure}
on \(L\) is a linear map \(\varphi\colon L\to L\) satisfying
\[
 [[x,y],\varphi(z)]+[[z,x],\varphi(y)]+[[y,z],\varphi(x)]=0
 \qquad (x,y,z\in L).
\]
We write \(\HomLie(L)\) for the vector space of all such maps.  Every scalar
multiple of the identity belongs to \(\HomLie(L)\).  The problem is to decide
whether these exhaust \(\HomLie(L)\) and, if not, to determine the whole
space.  

Hom-type identities arose from the study of \(q\)-deformations and
\(\sigma\)-derivations by Hartwig, Larsson and Silvestrov~\cite{HLS}; their
general theory was subsequently developed by Makhlouf and
Silvestrov~\cite{MS}.  For a fixed Lie bracket, the Hom--Jacobi identity is
linear in \(\varphi\), so \(\HomLie(L)\) is a linear subspace of
\(\End_{\K}(L)\).  Since scalar maps are always present, the substantive
question is whether \(L\) admits a non-scalar Hom--Lie structure and, if so,
to determine all such structures.  This has been one of the basic
classification problems in the subject from its early development.

For finite-dimensional semisimple Lie algebras in characteristic zero, Jin
and Li first considered Hom--Lie structures under the additional assumption
that \(\varphi\) preserve the bracket~\cite{JL}.  Xie, Jin and Liu removed
that assumption~\cite{XJL}.  Xie and Liu treated simple graded Lie algebras
of finite growth~\cite{XL}, while Makhlouf and Zusmanovich gave new proofs for
simple Lie algebras of classical type and extended the problem to current and
affine Kac--Moody algebras
\cite[Theorems~1--3]{MakhloufZusmanovich2018}.  In particular, every Hom--Lie
structure on a finite-dimensional simple Lie algebra in characteristic zero,
other than \(\mathfrak{sl}_2\), is scalar.  Makhlouf and Zusmanovich also
conjectured that, over a field of characteristic different from \(2\) and
\(3\), the only further finite-dimensional simple examples are the
Zassenhaus algebras
\cite[Conjecture following Theorem~1]{MakhloufZusmanovich2018}. 

In positive characteristic, Wu and Liu proved that every non-zero
bracket-preserving Hom--Lie structure on a graded Cartan-type algebra is the
identity~\cite{WuLiu2015}.  The unrestricted problem is substantially
different: a Hom--Lie structure need not preserve the bracket or the standard
grading.  Moreover, the classification of finite-dimensional simple Lie
algebras in characteristic \(p>3\) includes filtered Cartan-type algebras and,
when \(p=5\), the Melikian algebras
\cite{BlockWilson1988,PremetStrade2007,PremetStrade2008}; see also
\cite{PremetStrade2006}.

We resolve this classification problem for finite-dimensional simple Lie algebras over an
algebraically closed field of characteristic different from is not \(2\) or \(3\).
More precisely, we prove the conjecture of Makhlouf and Zusmanovich and
determine the Hom--Lie space in each exceptional case.  We use the ordered basis
\(\mathbf e=(e_-,h,e_+)\) of \(\mathfrak{sl}_2\), normalised by
\([e_-,h]=-e_-\), \([e_+,h]=e_+\), and \([e_-,e_+]=h\).  When
\(\operatorname{char}\K=p>3\), let \(\OO(1;n)\) be the divided-power algebra
with basis \(\mathcal B_n=(x^{(0)},\ldots,x^{(p^n-1)})\) and multiplication
\[
 x^{(a)}x^{(b)}
 =
 \begin{cases}
  \binom{a+b}{a}x^{(a+b)},&a+b<p^n,\\
  0,&a+b\geqslant p^n.
 \end{cases}
\]
Let \(\partial(x^{(a)})=x^{(a-1)}\), where \(x^{(-1)}=0\).  The Zassenhaus
algebra is \(W(1;n)=\OO(1;n)\partial\).  For \(u\in\OO(1;n)\), denote by
\(M_u\) the endomorphism of \(W(1;n)\) given by
\(M_u(f\partial)=uf\partial\).

\begin{theorem}\label{thm:main}
Let \(\K\) be an algebraically closed field with
\(\operatorname{char}\K\neq2,3\), and let \(L\) be a finite-dimensional simple
Lie algebra over \(\K\).  After choosing an isomorphism in \textup{(i)} or
\textup{(ii)}, identify the corresponding endomorphism spaces by conjugation.
Then the following hold.
\begin{enumerate}
\item[\textup{(i)}]
If \(L\simeq\mathfrak{sl}_2\), then, with respect to the ordered basis
\(\mathbf e=(e_-,h,e_+)\),
\[
 \HomLie(L)=
 \left\{
 \left.
 \begin{pmatrix}
  a&b&c\\ d&\eta&b\\ e&d&a
 \end{pmatrix}
 \,\right|\,
 a,b,c,d,e,\eta\in\K
 \right\}.
\]
\item[\textup{(ii)}]
If \(\operatorname{char}\K=p>3\) and \(L\simeq W(1;n)\) for some \(n\geqslant1\),
then
\[
 \HomLie(L)=\{M_u\mid u\in\OO(1;n)\}.
\]
The map \(u\mapsto M_u\) is an algebra isomorphism from \(\OO(1;n)\) onto this
space, with composition as multiplication.  In particular,
\(\dim_{\K}\HomLie(L)=p^n\).
\item[\textup{(iii)}]
In every other case, \(\HomLie(L)=\K\id_L\).
\end{enumerate}
\end{theorem}

Thus \(\dim_{\K}\HomLie(L)>1\) precisely when \(L\) is isomorphic to
\(\mathfrak{sl}_2\) or to a Zassenhaus algebra.  In the latter case,
composition identifies the Hom--Lie space with the divided-power coefficient
algebra itself.

The classification for simple classical Lie algebras in characteristic zero
is known \cite[Theorem~1]{MakhloufZusmanovich2018}.  The scalar conclusion
outside \(\mathfrak{sl}_2\) was also stated there in characteristic different
from \(2\) and \(3\), but the published arguments do not, as written, cover
every modular Chevalley algebra.  We give an independent root-space proof in
positive characteristic and explain the point in
Remark~\ref{rem:classical-positive-characteristic}.

For the classical algebras, we calculate
\(\HomLie(\mathfrak{sl}_2)\) explicitly and use the Chevalley root-space
relations in every other type.  For the Zassenhaus algebras, a direct
verification together with a theorem of Zusmanovich determines the whole
Hom--Lie space.  Low-degree relations and centraliser calculations give the
scalar conclusion for the remaining graded Cartan-type algebras and for the
Melikian algebras.  We then pass to filtered Cartan type and invoke the
Block--Wilson--Premet--Strade classification.

Unless stated otherwise, \(\K\) is algebraically closed and
\(\operatorname{char}\K\neq2,3\).  In positive characteristic we write
\(\operatorname{char}\K=p>3\).  All tensor products, Hom spaces and
endomorphism spaces are taken over \(\K\).  For \(X\subseteq L\), we write
\(\Cent_L(X)=\{y\in L\mid[y,X]=0\}\).  If \(A\) is a Lie algebra, then
\(A^{(0)}=A\) and \(A^{(i+1)}=[A^{(i)},A^{(i)}]\).

The rest of the paper is organised as follows.
In Section~\ref{sec:preliminaries}, we establish the module action on
\(\HomLie(L)\) and the principal-symbol lemma.  In
Section~\ref{sec:classical}, we treat the classical algebras, including the
six-dimensional \(\mathfrak{sl}_2\) case.  In
Section~\ref{sec:zassenhaus}, we determine the Hom--Lie space of each
Zassenhaus algebra.  In Section~\ref{sec:cartan}, we treat the remaining
graded Cartan-type families.  In Section~\ref{sec:melikian}, we consider the
Melikian algebras.  Finally, in Section~\ref{sec:filtered-cartan}, we pass to
filtered Cartan type and complete the proof of Theorem~\ref{thm:main}. 
As a formal consequence of scalar extension, we also obtain the corresponding result for finite-dimensional absolutely simple Lie algebras over arbitrary ground fields; see Corollary~\ref{cor:arbi-fields}.

\section{General results on Hom--Lie structures}\label{sec:preliminaries}

We record two facts used throughout the paper.  The space of Hom--Lie
structures is an \(L\)-submodule of \(\End_{\K}(L)\), and the principal symbol
of a non-zero Hom--Lie structure on a filtered algebra is a non-zero Hom--Lie
structure on its associated graded algebra.  We also prove a vanishing lemma
used in several of the Cartan-type arguments.

For \(\varphi\in\End_{\K}(L)\), write
\[
    J_\varphi(x,y,z)
    =
    [[x,y],\varphi(z)]
    +
    [[z,x],\varphi(y)]
    +
    [[y,z],\varphi(x)].
\]
Thus \(\varphi\in\HomLie(L)\) precisely when \(J_\varphi=0\).

\subsection{The module structure on \texorpdfstring{\(\HomLie(L)\)}{HomLie(L)}}

We recall an observation of Makhlouf and Zusmanovich
\cite[Lemma~1]{MakhloufZusmanovich2018}.
For \(\varphi\in\End_{\K}(L)\) and \(h\in L\), define
\begin{equation}\label{eq:bullet-action}
    (\varphi\bullet h)(x)
    =
    [\varphi(x),h]-\varphi([x,h]),
    \qquad x\in L.
\end{equation}
This is the natural right \(L\)-module structure on
\(\Hom_{\K}(L,L)\) induced by the adjoint representation.

\begin{lemma}\label{lem:HomLie-module}
Let \(L\) be a Lie algebra over \(\K\).
Then \(\HomLie(L)\) is stable under the action~\eqref{eq:bullet-action}.
Equivalently, if \(\varphi\in\HomLie(L)\) and \(h\in L\), then
\(\varphi\bullet h\in\HomLie(L)\).
\end{lemma}

\begin{proof}
Let \(F\colon L^{\otimes 3}\to L\) be a trilinear map.
The induced right action on such maps is
\begin{equation*}
    (F\bullet h)(x,y,z)
    =
    [F(x,y,z),h]
    -F([x,h],y,z)
    -F(x,[y,h],z)
    -F(x,y,[z,h]).
\end{equation*}
A direct expansion, using only the Jacobi identity in \(L\), gives
\(J_{\varphi\bullet h}=J_\varphi\bullet h\).  Hence the assignment
\(\varphi\mapsto J_\varphi\) from \(\End_{\K}(L)\) to
\(\Hom_{\K}(L^{\otimes3},L)\) is a homomorphism of right \(L\)-modules.  If
\(J_\varphi=0\), then \(J_{\varphi\bullet h}=0\), as required.
\end{proof}

The module structure will be used in the classical and Melikian arguments.  The
following consequence of the Hom--Jacobi identity will conclude several of
the Cartan-type arguments.

\begin{lemma}\label{lem:propagation}
Let \(L\) be a Lie algebra, and let \(V,A\subseteq L\) be subspaces such that
\([V,A]=V\) and
\(\Cent_L(V)\cap\Cent_L([A,A])=0\).
If \(\psi\in\HomLie(L)\) vanishes on \(V+A\), then \(\psi=0\).
\end{lemma}

\begin{proof}
For \(v\in V\), \(a\in A\), and \(y\in L\), the Hom--Jacobi identity gives
\([[v,a],\psi(y)]=0\).  Since \([V,A]=V\), it follows that
\(\psi(y)\in\Cent_L(V)\).  Applying the same identity to
\(a_1,a_2,y\), where \(a_1,a_2\in A\), gives
\(\psi(y)\in\Cent_L([A,A])\).  The last hypothesis now yields \(\psi(y)=0\).
\end{proof}

\subsection{Filtered Lie algebras and principal symbols}

We now record the filtered argument used later for filtered Cartan-type algebras.
Let \(L\) be a finite-dimensional Lie algebra equipped with a finite descending filtration
\[
    L=L_{(q)}\supseteq L_{(q+1)}
      \supseteq \cdots \supseteq L_{(s)}
      \supseteq L_{(s+1)}=0
\]
satisfying
\begin{equation}\label{eq:filtered-bracket}
    [L_{(i)},L_{(j)}]\subseteq L_{(i+j)}
    \qquad\text{for all } i,j .
\end{equation}
We use the convention that \(L_{(i)}=L\) for \(i\leqslant q\) and
\(L_{(i)}=0\) for \(i>s\).
Hereafter, a \emph{filtered Lie algebra} is a Lie algebra equipped with such a
filtration.  We always impose the stated boundary conventions.
The \emph{associated graded Lie algebra} is
\[
    \gr L=\bigoplus_i \gr_i L,
    \qquad
    \gr_i L=L_{(i)}/L_{(i+1)}.
\]
For \(x\in L_{(i)}\) and \(y\in L_{(j)}\), its bracket is
\[
    [x+L_{(i+1)},y+L_{(j+1)}]
    =
    [x,y]+L_{(i+j+1)}.
\]
The filtration condition makes this definition independent of the chosen
representatives.

Let \(0\neq\varphi\in\End_{\K}(L)\).
Its \emph{filtration degree} is the largest integer \(r\) such that
\[
    \varphi(L_{(i)})\subseteq L_{(i+r)}
    \qquad\text{for every }i.
\]
Such an integer exists: the set of integers with this property is non-empty
because \(L_{(j)}=L\) for \(j\leqslant q\), and it is bounded above because
\(\varphi\neq0\) and \(L_{(j)}=0\) for \(j>s\).
The \emph{principal symbol} of \(\varphi\) is the homogeneous map
\(\gr_r\varphi\colon\gr L\to\gr L\) of degree \(r\), defined by
\begin{equation}\label{eq:principal-symbol}
    \gr_r\varphi(x+L_{(i+1)})
    =
    \varphi(x)+L_{(i+r+1)},
    \qquad x\in L_{(i)}.
\end{equation}
The definition is independent of the representative \(x\), because \(\varphi(L_{(i+1)})\subseteq L_{(i+r+1)}\).
By the maximality of \(r\), the map \(\gr_r\varphi\) is non-zero.

\begin{lemma}\label{lem:symbol-HomLie}
Let \(L\) be a filtered Lie algebra over \(\K\).
If \(0\neq\varphi\in\HomLie(L)\) has filtration degree \(r\), then \(\gr_r\varphi\in\HomLie(\gr L)\).
\end{lemma}

\begin{proof}
Take homogeneous elements
\[
    \bar x=x+L_{(i+1)},\qquad
    \bar y=y+L_{(j+1)},\qquad
    \bar z=z+L_{(k+1)}
\]
of \(\gr L\), with \(x\in L_{(i)}\), \(y\in L_{(j)}\), and \(z\in L_{(k)}\).
Since \(\varphi\) has filtration degree \(r\), all three terms in the
Hom--Jacobi identity belong to \(L_{(i+j+k+r)}\).  Their sum is zero.
Passing to the homogeneous component of degree \(i+j+k+r\), and using
\eqref{eq:filtered-bracket} and \eqref{eq:principal-symbol}, we obtain
\[
\begin{split}
    [[\bar x,\bar y],(\gr_r\varphi)(\bar z)]
    +
    [[\bar z,\bar x],(\gr_r\varphi)(\bar y)]
    +
    [[\bar y,\bar z],(\gr_r\varphi)(\bar x)]
    =0 .
\end{split}
\]
Therefore \(\gr_r\varphi\) satisfies the Hom--Jacobi identity on \(\gr L\).
\end{proof}

\begin{lemma}\label{lem:filtered}
Let \(L\) be a filtered Lie algebra over \(\K\).
If \(\HomLie(\gr L)=\K\id_{\gr L}\), then
\[
    \HomLie(L)=\K\id_L.
\]
\end{lemma}

\begin{proof}
Let \(\varphi\in\HomLie(L)\).
If \(\varphi=0\), there is nothing to prove.
Assume that \(\varphi\neq0\), and let \(r\) be its filtration degree.
By Lemma~\ref{lem:symbol-HomLie}, the principal symbol
\(\gr_r\varphi\) lies in \(\HomLie(\gr L)\).
Hence \(\gr_r\varphi=\lambda\id_{\gr L}\) for some \(\lambda\in\K\).

The map \(\gr_r\varphi\) is non-zero and homogeneous of degree \(r\), whereas
\(\id_{\gr L}\) is homogeneous of degree \(0\).
Therefore \(r=0\).
The equality \(\gr_0\varphi=\lambda\id_{\gr L}\) means precisely that
\[
    (\varphi-\lambda\id_L)(L_{(i)})\subseteq L_{(i+1)}
    \qquad\text{for all }i.
\]
Let \(\psi=\varphi-\lambda\id_L\).
Since scalar multiples of the identity are Hom--Lie structures, we have \(\psi\in\HomLie(L)\).
If \(\psi\neq0\), then its filtration degree is at least \(1\).
Applying Lemma~\ref{lem:symbol-HomLie} to \(\psi\), we obtain a non-zero
principal symbol in \(\HomLie(\gr L)\) which is homogeneous of positive
degree.
This is impossible, because \(\HomLie(\gr L)=\K\id_{\gr L}\) contains no non-zero homogeneous map of positive degree.
Hence \(\psi=0\), and \(\varphi=\lambda\id_L\).
\end{proof}

Thus the scalar conclusion passes from the associated graded algebra to the
filtered algebra.

\section{Classical type and the three-dimensional case}\label{sec:classical}

We first determine the Hom--Lie structures on the three-dimensional simple
Lie algebra.  We then give a root-space proof for all remaining classical
simple Lie algebras.

\subsection{The algebra \texorpdfstring{\(\mathfrak{sl}_2\)}{sl2}}

We use the following normalisation of the standard ordered basis
\(\mathbf e=(e_-,h,e_+)\) of \(\mathfrak{sl}_2\):
\begin{equation}\label{eq:sl2-brackets}
    [e_-,h]=-e_-,
    \qquad
    [e_+,h]=e_+,
    \qquad
    [e_-,e_+]=h .
\end{equation}
With this convention, the non-scalar Hom--Lie structures can be written
explicitly as follows.  This calculation appears in
\cite[Proposition~1]{MakhloufZusmanovich2018}; we reproduce it in our
normalisation and without invoking any representation-theoretic
semisimplicity.

\begin{proposition}\label{prop:sl2}
Let \(\K\) be an algebraically closed field with
\(\operatorname{char}\K\neq2,3\), and let \(L\) have ordered basis
\(\mathbf e=(e_-,h,e_+)\) and brackets
\([e_-,h]=-e_-\), \([e_+,h]=e_+\), and \([e_-,e_+]=h\).  Then
\[
 \HomLie(L)=
 \left\{
 \left.
 \begin{pmatrix}
  a&b&c\\ d&\eta&b\\ e&d&a
 \end{pmatrix}
 \,\right|\,
 a,b,c,d,e,\eta\in\K
 \right\},
\]
where the matrices are taken with respect to \(\mathbf e\).  In particular,
\(\dim_{\K}\HomLie(L)=6\).
\end{proposition}

\begin{proof}
Write
\[
\begin{aligned}
 \varphi(e_-)&=a e_-+d h+e e_+,&
 \varphi(h)&=b e_-+\eta h+f e_+,&
 \varphi(e_+)&=c e_-+g h+k e_+.
\end{aligned}
\]
The map \(J_\varphi\) is alternating.  Since \(\dim L=3\), it is enough to
impose \(J_\varphi(e_-,h,e_+)=0\).  Using
\eqref{eq:sl2-brackets}, we obtain
\[
 J_\varphi(e_-,h,e_+)=(g-b)e_-+(a-k)h+(f-d)e_+.
\]
Thus \(g=b\), \(k=a\), and \(f=d\), while the remaining six coefficients are
arbitrary.
\end{proof}

\subsection{Classical simple Lie algebras other than
\texorpdfstring{\(\mathfrak{sl}_2\)}{sl2}}

We use the Chevalley root-space description; in positive characteristic our
conventions follow \cite[Chapter~4, Section~4.1]{StradeI}.  Let \(R\) be an
irreducible reduced root system with simple roots
\(\alpha_1,\ldots,\alpha_\ell\).  Choose a Chevalley basis
\[
    \mathcal B_R
    =
    \{h_1,\ldots,h_\ell\}\cup\{e_\alpha\mid\alpha\in R\},
\]
normalised by
\[
\begin{aligned}
    \relax [h_i,h_j]&=0,
    & [h_i,e_\alpha]&=\langle\alpha,\alpha_i^\vee\rangle e_\alpha,\\
    [e_\alpha,e_{-\alpha}]&=h_\alpha,
    & [e_\alpha,e_\beta]&=N_{\alpha,\beta}e_{\alpha+\beta}
       &&(\alpha+\beta\in R),
\end{aligned}
\]
and \([e_\alpha,e_\beta]=0\) if \(\beta\neq-\alpha\) and
\(\alpha+\beta\notin R\).  The integral Chevalley constants satisfy
\(|N_{\alpha,\beta}|=r_{\alpha,\beta}+1\), where
\(r_{\alpha,\beta}\) is the largest non-negative integer such that
\(\beta-r_{\alpha,\beta}\alpha\in R\).

Let \(\mathfrak g_{\mathbb Z}\) be the \(\mathbb Z\)-span of this basis and
\(\mathfrak g_{\K}=\mathfrak g_{\mathbb Z}\otimes_{\mathbb Z}\K\).  By a
simple Lie algebra of classical type we mean the simple algebra obtained from
\(\mathfrak g_{\K}^{(1)}\) by factoring out its centre.  If
\(\operatorname{char}\K\neq2,3\), a central quotient occurs only in type
\(A_\ell\) when \(p\mid(\ell+1)\).  Its root spaces remain one-dimensional,
and the required Chevalley relations descend to the quotient.  The possible
types are \(A_\ell,B_\ell,C_\ell,D_\ell,G_2,F_4,E_6,E_7\), and \(E_8\).

\begin{lemma}\label{lem:classical-root-facts}
Let \(L\) be a finite-dimensional simple Lie algebra of classical type over an
algebraically closed field \(\K\) of characteristic \(0\) or \(p>3\), and suppose
that \(L\not\simeq\mathfrak{sl}_2\).  Choose a Chevalley root decomposition
\[
   L=H\oplus\bigoplus_{\alpha\in R}\K e_\alpha,
   \qquad [h,e_\alpha]=\alpha(h)e_\alpha ;
\]
here \(R\) denotes the underlying irreducible reduced root system and, in a
simple central quotient, also its induced set of root functions on \(H\).
Then the following hold.
\begin{enumerate}
\item If \(\alpha,\beta\in R\) and \(\alpha\neq\pm\beta\), then \(\alpha\) and
      \(\beta\), regarded as elements of \(H^*\), are linearly independent.
\item For every \(\beta\in R\), there exists \(\delta\in R\) such that
      \(\delta\neq\pm\beta\) and \(\delta-\beta\in R\).
\item If \(\alpha,\beta,\alpha+\beta\in R\), then
      \([e_\alpha,e_\beta]\neq0\).
\end{enumerate}
\end{lemma}

\begin{proof}
The last assertion follows from the Chevalley commutator formula: the relevant
structure constants have absolute value at most \(3\), and hence remain non-zero
in characteristic \(p>3\).

For the first assertion, suppose that \(\alpha=c\beta\) in \(H^*\) for some
\(c\in\K^\times\).  Let
\(h_\gamma=[e_\gamma,e_{-\gamma}]\) for \(\gamma\in R\).  Since \(\alpha(h_\alpha)=\beta(h_\beta)=2\), one has
\[
   \alpha(h_\beta)=2c,
   \qquad
   \beta(h_\alpha)=2c^{-1},
\]
and hence
\[
   \alpha(h_\beta)\,\beta(h_\alpha)=4.
\]
Because \(\alpha\neq\pm\beta\), the two roots are non-proportional in the
underlying reduced root system, and the product of the corresponding Cartan
integers is one of \(0,1,2,3\).  These values are distinct from \(4\) in
characteristic \(0\) and in characteristic \(p>3\), a contradiction.

For the second assertion, choose a simple system containing \(\beta\).
This is possible for every root.  Since the reduced root system is irreducible
and has rank greater than one, \(\beta\) has an adjacent simple root \(\eta\).
Then \(\beta+\eta\) is a root, so \(\delta=\beta+\eta\) has all the required
properties.
\end{proof}

\begin{theorem}\label{thm:classical}
Let \(L\) be a finite-dimensional simple Lie algebra of classical type over
\(\K\), where \(\K\) is algebraically closed and
\(\operatorname{char}\K\neq2,3\).  If
\(L\not\simeq\mathfrak{sl}_2\), then
\[
   \HomLie(L)=\K\id_L.
\]
\end{theorem}

\begin{proof}
The standard left action of \(L\) on \(\End_{\K}(L)\) is
\[
   (h\cdot\varphi)(x)=[h,\varphi(x)]-\varphi([h,x]).
\]
It is the negative of the right-adjoint action in
Lemma~\ref{lem:HomLie-module}; hence \(\HomLie(L)\) is stable under this action as
well.  Since \(H\) acts diagonally on \(L\), it acts diagonally on
\(\End_{\K}(L)\).  It is therefore enough to determine a homogeneous
\(\varphi\in\HomLie(L)\) of \(H\)-weight \(\sigma\).  Thus
\[
   \varphi(L_\mu)\subseteq L_{\mu+\sigma},
   \qquad L_0=H,\qquad L_\alpha=\K e_\alpha\quad(\alpha\in R),
\]
where \(L_\nu=0\) if \(\nu\) is neither \(0\) nor a root.
Here \(\sigma\) is a weight of \(\End_{\K}(L)\), hence a difference of two
weights of \(L\); at this stage it is not assumed to be a root.

We first show that every component of non-zero weight vanishes.  Assume that
\(\sigma\neq0\).

Suppose first that \(\varphi(H)\neq0\).  Then \(\sigma\in R\), and there is a
non-zero \(f\in H^*\) such that \(\varphi(h)=f(h)e_\sigma\) for \(h\in H\).
If \(\alpha\in R\) and \([e_\alpha,e_\sigma]\neq0\), the Hom--Jacobi identity for
\(h,k,e_\alpha\) gives
\[
   \bigl(\alpha(k)f(h)-\alpha(h)f(k)\bigr)
   [e_\alpha,e_\sigma]=0.
\]
Taking \(\alpha=-\sigma\) shows that \(f\) is proportional to \(\sigma\).
Apply Lemma~\ref{lem:classical-root-facts}(2) to \(\sigma\), and write
\(\alpha=\delta-\sigma\).  Then \(\alpha\) is not proportional to \(\sigma\) and
\(\alpha+\sigma=\delta\in R\).  The same identity
then says that \(f\) is proportional to \(\alpha\), a contradiction.  Hence
\(\varphi(H)=0\).

Assume now that \(\varphi(e_\beta)\neq0\) for some \(\beta\in R\), and write
\(\gamma=\beta+\sigma\).

If \(\gamma\in R\) and \(\gamma\neq-\beta\), let \(\alpha=-\gamma\).
The equality \(\alpha=\beta\) is excluded by \(\gamma\neq-\beta\), while
\(\alpha=-\beta\) would imply \(\sigma=0\).  Thus \(\alpha\neq\pm\beta\), and
Lemma~\ref{lem:classical-root-facts}(1) allows us to choose \(h\in H\) with
\(\beta(h)=0\) and \(\alpha(h)\neq0\).  Since
\(\varphi(H)=0\), the Hom--Jacobi identity for \(h,e_\alpha,e_\beta\) reduces to
\(\alpha(h)[e_\alpha,\varphi(e_\beta)]=0\).
But \(\varphi(e_\beta)\) is a non-zero multiple of \(e_\gamma\), and
\([e_{-\gamma},e_\gamma]\neq0\), a contradiction.

If \(\gamma=-\beta\), choose, by
Lemma~\ref{lem:classical-root-facts}, a root \(\delta\) such that
\(\delta\neq\pm\beta\) and \(\delta-\beta\in R\).  Take \(h\in H\) with
\(\beta(h)=0\) and \(\delta(h)\neq0\).  The Hom--Jacobi identity for
\(h,e_\delta,e_\beta\) gives
\(\delta(h)[e_\delta,\varphi(e_\beta)]=0\),
whereas \(\varphi(e_\beta)\) is a non-zero multiple of \(e_{-\beta}\) and
\([e_\delta,e_{-\beta}]\neq0\).  This is again a contradiction.

It remains to consider \(\gamma=0\).  Then \(\sigma=-\beta\) and
\(\varphi(e_\beta)=t\in H\).  The weight \(-2\beta\) is non-zero.  If it were a
root, Lemma~\ref{lem:classical-root-facts}(1) would force it to be either
\(\beta\) or \(-\beta\).  The first equality gives \(3\beta=0\), and the
second gives \(\beta=0\), both impossible because
\(\operatorname{char}\K\neq2,3\).  Hence
\(\varphi(e_{-\beta})=0\).  The
Hom--Jacobi identity for \(h,e_{-\beta},e_\beta\) yields
\(\beta(t)=0\).  If \(\alpha\neq\pm\beta\), choose \(h\in H\) with
\(\beta(h)=0\) and \(\alpha(h)\neq0\).  The identity for
\(h,e_\alpha,e_\beta\) gives \(\alpha(t)=0\).  Thus every root vanishes on \(t\).
Consequently \(t\) is central in \(L\), and so \(t=0\), another contradiction.

We have proved that every non-zero weight component of \(\HomLie(L)\) is zero.
It remains to consider a weight-zero map \(\varphi\).  Write
\[
   \varphi(H)\subseteq H,
   \qquad
   \varphi(e_\alpha)=c_\alpha e_\alpha
   \quad(\alpha\in R).
\]
For \(h,k\in H\), the Hom--Jacobi identity for \(h,k,e_\alpha\) is
\[
   \alpha(h)\alpha(\varphi(k))
   =
   \alpha(k)\alpha(\varphi(h)).
\]
It follows that \(\alpha\circ\varphi|_H=\lambda_\alpha\alpha\) for some
\(\lambda_\alpha\in\K\).  If \(\alpha\) and \(\beta\) are adjacent
simple roots, then \(\alpha+\beta\) is a root.  Comparing the preceding relation
for \(\alpha\), \(\beta\), and \(\alpha+\beta\), and using the linear
independence of \(\alpha\) and \(\beta\), gives
\(\lambda_\alpha=\lambda_\beta=\lambda_{\alpha+\beta}\).
The Dynkin diagram is connected, so all simple roots have the same eigenvalue,
say \(\lambda\).  The simple roots span \(H^*\): an element of \(H\) annihilated by
all roots would be central in \(L\).  Hence
\(\varphi|_H=\lambda\id_H\).

The Hom--Jacobi identity for \(h,e_\alpha,e_{-\alpha}\) gives
\(c_{-\alpha}=c_\alpha\).
If \(\alpha,\beta,\alpha+\beta\in R\), the identity for
\(h,e_\alpha,e_\beta\) gives, as an equality in \(H^*\),
\[
   (c_\beta-\lambda)\alpha
   +(c_\alpha-\lambda)\beta=0.
\]
Since \(\alpha\) and \(\beta\) are linearly independent, this implies
\(c_\alpha=c_\beta=\lambda\).  Every simple root has an adjacent simple root,
so this holds for all simple roots.  If \(\gamma\) is a positive non-simple root,
choose a simple root \(\alpha\) such that \(\gamma-\alpha\) is a root.  Applying the
same identity to \(\gamma\) and \(-\alpha\) gives \(c_\gamma=\lambda\).  Finally,
\(c_{-\gamma}=c_\gamma\), so \(c_\gamma=\lambda\) for every root \(\gamma\).
Therefore \(\varphi=\lambda\id_L\), as required.
\end{proof}

\begin{remark}\label{rem:classical-positive-characteristic}
The characteristic-zero part of Theorem~\ref{thm:classical} is known
\cite[Theorem~1]{MakhloufZusmanovich2018}.  Some care is required before the
arguments given there can be used in positive characteristic.

The direct calculation for \(\mathfrak{sl}_2\) remains valid, but the ensuing
module decomposition uses complete reducibility of finite-dimensional
\(\mathfrak{sl}_2\)-modules
\cite[proof of Proposition~1]{MakhloufZusmanovich2018}, which fails in general
in positive characteristic.  In the first proof of the higher-rank case, a
weight of \(\End_{\K}(L)\) is treated as a root before this has been
established, and an arbitrary line in the Cartan subalgebra is treated as an
intersection of root hyperplanes
\cite[first proof of Theorem~1]{MakhloufZusmanovich2018}.  The proof above
avoids both points: it first determines when the weight is a root and then
uses adjacent roots.

The alternative proof proceeds through invariant forms and quasiderivations
\cite[second proof of Theorem~1]{MakhloufZusmanovich2018}.  The relevant
positive-characteristic results require a non-singular trace form
\cite[Corollary~4.14 and Theorem~5.23]{LegerLuks2000}, a hypothesis that is
not uniform in the modular cases; for example, the Killing form of
\(\mathfrak{psl}_n\) vanishes when \(p\mid n\).  The root-space argument above
therefore supplies the required proof for all \(p>3\), including the simple
central quotients of type \(A\).
\end{remark}

Proposition~\ref{prop:sl2} and Theorem~\ref{thm:classical} complete the
classical case.  Since every finite-dimensional simple Lie algebra in
characteristic zero is classical, they also prove the characteristic-zero
part of Theorem~\ref{thm:main}.

It remains to consider positive characteristic.  From now on, unless
explicitly stated otherwise, we assume that \(\operatorname{char}\K=p>3\).
The classical simple Lie algebras have already been dealt with by
Theorem~\ref{thm:classical}.  The rest of the proof concerns the Cartan-type and
Melikian algebras occurring in the Block--Wilson--Premet--Strade
classification; see
\cite{BlockWilson1988,PremetStrade2007,PremetStrade2008}.

\section{The Zassenhaus algebras}\label{sec:zassenhaus}

We next determine the second exceptional Hom--Lie space.  Throughout this
section, \(\operatorname{char}\K=p>3\) and \(n\geqslant1\).  Our conventions for
divided powers and special derivations agree with
\cite[Introduction, pages~2--4, and Chapter~2, Section~2.1]{StradeI}.

The truncated divided-power algebra \(\OO(1;n)\) has ordered basis
\(\mathcal B_n=(x^{(a)})_{0\leqslant a<p^n}\), identity
\(x^{(0)}=1\), and multiplication
\[
 x^{(a)}x^{(b)}
 =
 \begin{cases}
  \binom{a+b}{a}x^{(a+b)},&a+b<p^n,\\
  0,&a+b\geqslant p^n.
 \end{cases}
\]
The binomial coefficients are read in \(\K\).  Let \(\partial\) be the special
derivation determined by \(\partial(x^{(a)})=x^{(a-1)}\), where
\(x^{(-1)}=0\), and write \(f'=\partial(f)\).  The Zassenhaus algebra is
\(W(1;n)=\OO(1;n)\partial\), with bracket
\begin{equation}\label{eq:zassenhaus-bracket}
 [f\partial,g\partial]=(fg'-gf')\partial
 \qquad(f,g\in\OO(1;n)).
\end{equation}

We shall use a result for group-graded generalised Witt algebras.  If
\(G\subseteq\K\) is an additive subgroup, then \(W_G\) has basis
\(\mathcal B_G=\{e_\alpha\mid\alpha\in G\}\) and bracket
\[
 [e_\alpha,e_\beta]=(\beta-\alpha)e_{\alpha+\beta}
 \qquad(\alpha,\beta\in G).
\]
For \(\sigma\in G\), let \(T_\sigma(e_\alpha)=e_{\alpha+\sigma}\).
Zusmanovich proved that
\[
 \HomLie(W_G)=\Span_{\K}\{T_\sigma\mid\sigma\in G\}\simeq\K[G]
\]
\cite[Theorem~1]{Zusmanovich2023}.  When
\(G=(\mathbb F_{p^n},+)\), the algebra \(W_G\) is isomorphic to \(W(1;n)\);
see \cite[Section~3, proof of Proposition~4]{Zusmanovich2023}.  Consequently,
\begin{equation}\label{eq:zassenhaus-dimension}
 \dim_{\K}\HomLie(W(1;n))=p^n.
\end{equation}

For \(u\in\OO(1;n)\), let \(M_u\in\End_{\K}(W(1;n))\) be multiplication by
\(u\), so that \(M_u(f\partial)=uf\partial\).

\begin{proposition}\label{prop:zassenhaus}
Let \(\K\) be an algebraically closed field of characteristic \(p>3\), and let
\(n\geqslant1\).  Then
\[
 \HomLie(W(1;n))=\{M_u\mid u\in\OO(1;n)\}.
\]
Moreover, under composition, \(u\mapsto M_u\) is an algebra isomorphism.
\end{proposition}

\begin{proof}
Let \(f,g,k,u\in\OO(1;n)\).  By \eqref{eq:zassenhaus-bracket}, the coefficient
of \(\partial\) in the cyclic Hom--Jacobi expression associated with \(M_u\)
is
\[
\begin{aligned}
 \sum_{\mathrm{cyc}}
 \bigl((fg'-gf')(uk)'-uk(fg'-gf')'\bigr)
 &=
 u\sum_{\mathrm{cyc}}
 \bigl((fg'-gf')k'-k(fg'-gf')'\bigr)\\
 &\quad
 +u'\bigl(k(fg'-gf')+g(kf'-fk')+f(gk'-kg')\bigr).
\end{aligned}
\]
The first cyclic sum vanishes by the Jacobi identity in \(W(1;n)\), and the
six terms in parentheses cancel in pairs.  Thus \(M_u\in\HomLie(W(1;n))\).

The map \(u\mapsto M_u\) is injective because
\(M_u(\partial)=u\partial\).  Its image has dimension \(p^n\), and
\eqref{eq:zassenhaus-dimension} shows that it is all of
\(\HomLie(W(1;n))\).  Finally, \(M_uM_v=M_{uv}\) and \(M_1=\id_{W(1;n)}\),
which proves the final assertion.
\end{proof}

\section{Cartan-type algebras other than Zassenhaus}\label{sec:cartan}

We now prove that every Hom--Lie structure is scalar for the remaining
graded Cartan-type algebras.  The four arguments use their standard
realisations, but only a small number of low-degree components enter the
proofs.

\subsection{Witt algebras}\label{sec:witt}

We use the divided-power realisation of the Witt algebras, with the conventions
of \cite[Chapter~2, Section~2.1]{StradeI}.

Assume throughout this subsection that \(\operatorname{char}\K=p>3\) and
\(m\geqslant 2\).  Let
\(\uln=(n_1,\ldots,n_m)\in\mathbb Z_{>0}^m\), and let
\(\OO(m;\uln)\) be the
truncated divided-power algebra in the coordinate vector
\(\mathbf x=(x_1,\ldots,x_m)\).  Its distinguished basis is
\[
    \mathcal B_{\OO}
    =
    \{x^{(\alpha)}=x_1^{(\alpha_1)}\cdots x_m^{(\alpha_m)}
      \mid 0\leqslant\alpha_i<p^{n_i}\},
\]
and its multiplication is
\[
    x^{(\alpha)}x^{(\beta)}
    =
    \binom{\alpha+\beta}{\alpha}x^{(\alpha+\beta)},
    \qquad
    \binom{\alpha+\beta}{\alpha}
    =\prod_{i=1}^m\binom{\alpha_i+\beta_i}{\alpha_i},
\]
when \(\alpha_i+\beta_i<p^{n_i}\) for every \(i\), and the product is zero
otherwise.  The binomial coefficients are read in \(\K\), and multiplication
is extended \(\K\)-bilinearly.  This makes \(\OO(m;\uln)\) a commutative
associative algebra with identity \(x^{(0)}=1\).  We use the uniform convention
that \(x^{(\gamma)}=0\) whenever some component \(\gamma_i\) lies outside
\(0\leqslant\gamma_i<p^{n_i}\).  In particular,
\(\dim_{\K}\OO(m;\uln)=p^{n_1+\cdots+n_m}\).  We write
\(x_i=x^{(\mathbf e_i)}=x_i^{(1)}\).  For \(1\leqslant i\leqslant m\), let
\(\partial_i\) be the
standard special derivation determined by
\[
    \partial_i(x^{(\alpha)})=x^{(\alpha-\mathbf e_i)},
\]
where \(\mathbf e_i\) is the \(i\)-th standard basis vector of
\(\mathbb Z^m\), and the total degree of \(x^{(\alpha)}\) is
\(\lvert\alpha\rvert=\alpha_1+\cdots+\alpha_m\).

A derivation \(D\) of \(\OO(m;\uln)\) is called \emph{special} if
\[
    D(x^{(\alpha)})
    =
    \sum_{i=1}^m x^{(\alpha-\mathbf e_i)}D(x_i)
\]
for every multi-index \(\alpha\) in the stated range, with an out-of-range divided
power interpreted as zero.  Its values on \(x_1,\ldots,x_m\) determine it
uniquely, and every special derivation has a unique expression
\(D=\sum_{i=1}^m f_i\partial_i\), where \(f_i\in\OO(m;\uln)\).  Consequently,
the \emph{Cartan-type Witt algebra}, namely the Lie algebra of all special
derivations, is
\[
    W(m;\uln)=\bigoplus_{s=1}^m \OO(m;\uln)\partial_s .
\]
Its bracket is determined by
\[
    [f\partial_r,g\partial_s]
    =
    f\partial_r(g)\partial_s
    -
    g\partial_s(f)\partial_r ,
    \qquad f,g\in\OO(m;\uln).
\]
We use the standard grading \(W(m;\uln)=\bigoplus_{d\geqslant -1}W(m;\uln)_d\), where
\(f\partial_s\) has degree \(\deg f-1\) for homogeneous \(f\).  Let
\(L=W(m;\uln)\).  Its two lowest components are
\[
    V=L_{-1}=\Span_{\K}\{\partial_1,\ldots,\partial_m\},
    \qquad
    L_0=\Span_{\K}\{x_i\partial_j\mid 1\leqslant i,j\leqslant m\}
        \simeq\mathfrak{gl}_m .
\]

\begin{lemma}\label{lem:W-centralizer}
Let \(\K\) be an algebraically closed field of characteristic \(p>3\), let
\(m\geqslant2\), and let
\(\uln=(n_1,\ldots,n_m)\in\mathbb Z_{>0}^m\).  For
\(L=W(m;\uln)\) and
\(V=L_{-1}=\Span_{\K}\{\partial_1,\ldots,\partial_m\}\), one has
\[
    \Cent_L(V)=V .
\]
\end{lemma}

\begin{proof}
Let \(D\in\Cent_L(V)\), and write \(D=\sum_s f_s\partial_s\).  For every \(i\), the
identity \([\partial_i,D]=0\) gives
\(\sum_s \partial_i(f_s)\partial_s=0\).
Thus each coefficient \(f_s\) is killed by all partial derivatives
\(\partial_1,\ldots,\partial_m\).  The only simultaneous constants in
\(\OO(m;\uln)\) are the elements of \(\K\), and hence every \(f_s\) is scalar.
Therefore \(D\in V\).  The reverse inclusion is immediate, since the
\(\partial_i\) commute with one another.
\end{proof}

The following elementary low-degree spanning fact will be used twice.

\begin{lemma}\label{lem:W-L1-generates-L0}
Let \(\K\) be an algebraically closed field of characteristic \(p>3\), let
\(m\geqslant2\), let
\(\uln=(n_1,\ldots,n_m)\in\mathbb Z_{>0}^m\), and let
\(L=W(m;\uln)\) carry its standard grading.  Then, for every
\(1\leqslant i\leqslant m\),
\[
    [L_1,\partial_i]=L_0 .
\]
\end{lemma}

\begin{proof}
It suffices to obtain the basis elements \(x_q\partial_s\) of \(L_0\).  If
\(q\neq i\), then \(x_ix_q\partial_s\in L_1\) and
\([x_ix_q\partial_s,\partial_i]=-x_q\partial_s\).  If \(q=i\), then
\(x_i^{(2)}\partial_s\in L_1\) and
\([x_i^{(2)}\partial_s,\partial_i]=-x_i\partial_s\).
Hence every \(x_q\partial_s\) lies in \([L_1,\partial_i]\).
\end{proof}

\begin{theorem}\label{thm:W}
Let \(\K\) be an algebraically closed field of characteristic \(p>3\), let
\(m\geqslant2\), and let
\(\uln=(n_1,\ldots,n_m)\in\mathbb Z_{>0}^m\).  Let
\(L=W(m;\uln)\).  Then
\[
    \HomLie(L)=\K\id_L .
\]
\end{theorem}

\begin{proof}
Let \(\varphi\in\HomLie(L)\).  We prove that
\(\varphi\) is scalar.

\smallskip\noindent\textbf{Preservation of \(V\).}
Fix indices \(j\) and \(b\).
Choose \(i\neq j\), and let \(X=x_i\partial_b\in L_0\).  Then
\([X,\partial_i]=-\partial_b\) and \([\partial_j,X]=0\).  Applying the
Hom--Jacobi identity to \(\partial_i,\partial_j,X\), and using
\([\partial_i,\partial_j]=0\), gives
\([-\partial_b,\varphi(\partial_j)]=0\).
Since \(b\) and \(j\) are arbitrary, \(\varphi(\partial_j)\) centralises all of
\(V\).  By Lemma~\ref{lem:W-centralizer}, \(\varphi(\partial_j)\in V\), and hence
\(\varphi(V)\subseteq V\).

\smallskip\noindent\textbf{Scalar action on \(V\).}
Write \(\varphi(\partial_j)=\sum_r a_{rj}\partial_r\).
Let \(i\neq j\), and take \(X=x_ix_j\partial_b\in L_1\).  Then
\([X,\partial_i]=-x_j\partial_b\) and \([\partial_j,X]=x_i\partial_b\).  The
Hom--Jacobi identity for \(\partial_i,\partial_j,X\) yields
\begin{equation*}
    0
    =
    [-x_j\partial_b,\varphi(\partial_j)]
    +
    [x_i\partial_b,\varphi(\partial_i)]
    =
    (a_{jj}-a_{ii})\partial_b .
\end{equation*}
Thus all diagonal coefficients \(a_{ii}\) are equal.

To show that the off-diagonal coefficients vanish, take
\(X=x_i^{(2)}\partial_b\) with
\(i\neq j\).  Then \([X,\partial_i]=-x_i\partial_b\) and \([\partial_j,X]=0\).
The Hom--Jacobi identity gives
\(0=[-x_i\partial_b,\varphi(\partial_j)]=a_{ij}\partial_b\).
Hence \(a_{ij}=0\) whenever \(i\neq j\).  Therefore
\(\varphi|_V=\lambda\id_V\) for some \(\lambda\in\K\).

Let \(\psi=\varphi-\lambda\id_L\).  Then \(\psi\in\HomLie(L)\) and
\(\psi(V)=0\).  We shall show that \(\psi=0\).

\smallskip\noindent\textbf{Vanishing on \(L_0\).}
Let \(k=x_c\partial_d\) be a
basis element of \(L_0\).  Fix \(b\), choose \(i\neq c\), and let
\(h=x_i\partial_b\).  Then \([\partial_i,h]=\partial_b\) and
\([k,\partial_i]=0\).  Applying the Hom--Jacobi identity to
\(\partial_i,h,k\), and using \(\psi(\partial_i)=0\), gives
\([\partial_b,\psi(k)]=0\).
Since \(b\) is arbitrary, Lemma~\ref{lem:W-centralizer} implies
\(\psi(k)\in V\).

It remains to prove that this vector is zero.  Choose \(u=\partial_i\in V\) with
\(i\neq c\), so that \([u,k]=0\).  For \(Y\in L_1\), the Hom--Jacobi identity
applied to \(u,k,Y\) gives
\([[Y,u],\psi(k)]=0\),
because \(\psi(u)=0\) and \([u,k]=0\).  By
Lemma~\ref{lem:W-L1-generates-L0}, the elements \([Y,u]\), with \(Y\in L_1\),
span \(L_0\).  Hence \([L_0,\psi(k)]=0\).  But \(V\) has no non-zero
\(L_0\)-fixed vector.  Indeed, if \(v=\sum_r \mu_r\partial_r\in V\) and
\([x_a\partial_b,v]=0\) for all \(a,b\), then
\(-\mu_a\partial_b=0\) for all \(a,b\), so all \(\mu_a\) vanish.  Therefore
\(\psi(k)=0\).  Since the elements \(x_c\partial_d\) span \(L_0\), we have
\(\psi(L_0)=0\).

%\smallskip\noindent\textbf{Conclusion.}
The equality \([V,L_0]=V\) is immediate.  By
Lemma~\ref{lem:W-centralizer}, \(\Cent_L(V)=V\).  The off-diagonal elements
\(x_a\partial_b\), with \(a\neq b\), belong to \([L_0,L_0]\) and have no
common non-zero fixed vector in \(V\), since
\([x_a\partial_b,\sum_r\mu_r\partial_r]=-\mu_a\partial_b\).
Lemma~\ref{lem:propagation}, applied with \(A=L_0\), gives \(\psi=0\).
Therefore \(\varphi=\lambda\id_L\).
\end{proof}

\subsection{Special algebras}

Let \(\K\) be algebraically closed of
characteristic \(p>3\), let \(m\geqslant3\), and let
\(\uln=(n_1,\ldots,n_m)\in\mathbb Z_{>0}^m\).  We use the divided-power
algebra \(\OO(m;\uln)\), its special derivations \(\partial_i\), and the Witt
algebra \(W(m;\uln)\) defined in Section~\ref{sec:witt}.  Our conventions agree with
\cite[Chapter~6, Section~6.3]{StradeI}.  In the classification, the simple graded
Cartan-type algebras of type \(S\) occur for \(m\geqslant3\); the two-variable
symplectic family belongs to type \(H\).

The \emph{special algebra} \(S(m;\uln)\) is the subalgebra of
\(W(m;\uln)\) consisting of
divergence-free special derivations.  Thus, for
\(D=\sum_{s=1}^m f_s\partial_s\), one has
\[
    \diver(D)=\sum_{s=1}^m \partial_s(f_s),
    \qquad
    S(m;\uln)=\{D\in W(m;\uln)\mid \diver(D)=0\}.
\]
The identity
\(\diver([D,E])=D(\diver(E))-E(\diver(D))\) shows directly that this is a
Lie subalgebra.  It is graded by
\(S(m;\uln)_d=S(m;\uln)\cap W(m;\uln)_d\), and every derived algebra carries
the induced grading.
For \(1\leqslant i<j\leqslant m\) and \(f\in\OO(m;\uln)\), write
\[
    D_{ij}(f)
    =
    \partial_j(f)\partial_i-\partial_i(f)\partial_j .
\]
These derivations are divergence-free.  The first derived algebra
\(L=S(m;\uln)^{(1)}\) is spanned by the elements \(D_{ij}(f)\) and is the
simple special algebra
used below.  Its standard grading is induced from the Witt grading; equivalently,
\(D_{ij}(x^{(\alpha)})\) has degree \(|\alpha|-2\).  In particular,
\[
    \begin{aligned}
    L_{-1}&=V=\Span_{\K}\{\partial_1,\ldots,\partial_m\},\\
    L_0&=\Span_{\K}\{x_i\partial_j\mid i\neq j\}
        \oplus
        \Span_{\K}\{x_i\partial_i-x_j\partial_j\mid i\neq j\}
        \simeq \mathfrak{sl}_m .
    \end{aligned}
\]
In degrees \(-1\), \(0\), and \(1\), the homogeneous components of
\(S(m;\uln)^{(1)}\) agree with those of \(S(m;\uln)\).  These are the only
components used in the argument below.

\begin{lemma}\label{lem:S-centralizer}
Let \(\K\) be an algebraically closed field of characteristic \(p>3\), let
\(m\geqslant3\), and let
\(\uln=(n_1,\ldots,n_m)\in\mathbb Z_{>0}^m\).  Let
\(S(m;\uln)^{(1)}\subseteq G\subseteq S(m;\uln)\), and let
\(V=S(m;\uln)^{(1)}_{-1}\).  Then
\[
    \Cent_G(V)=V .
\]
\end{lemma}

\begin{proof}
Since \(G\subseteq W(m;\uln)\), we may compute the centraliser inside the Witt
algebra.  Let \(D=\sum_s f_s\partial_s\in G\) centralise every \(\partial_i\).
Then \([\partial_i,D]=\sum_s\partial_i(f_s)\partial_s=0\) for all \(i\).  Hence
each coefficient \(f_s\) is killed by all divided-power partial derivatives
\(\partial_1,\ldots,\partial_m\).  The only simultaneous constants in
\(\OO(m;\uln)\) are the scalars, so each \(f_s\in\K\), and therefore
\(D\in V\).  The reverse inclusion is immediate, since the \(\partial_i\) commute
with one another.
\end{proof}

\begin{lemma}\label{lem:S-L1-generates-L0}
Let \(\K\) be an algebraically closed field of characteristic \(p>3\), let
\(m\geqslant3\), let
\(\uln=(n_1,\ldots,n_m)\in\mathbb Z_{>0}^m\), and let
\(L=S(m;\uln)^{(1)}\) carry its standard grading.  Then, for every
\(1\leqslant i\leqslant m\),
\[
    [L_1,\partial_i]=L_0 .
\]
\end{lemma}

\begin{proof}
Fix \(i\).  We show that the standard generators of
\(L_0\simeq\mathfrak{sl}_m\) lie in \([L_1,\partial_i]\).

First consider an off-diagonal element \(x_q\partial_s\), where \(q\neq s\).
If \(q\neq i\) and \(s\notin\{i,q\}\), then
\(x_ix_q\partial_s\in L_1\) is divergence-free and
\([x_ix_q\partial_s,\partial_i]=-x_q\partial_s\).  If \(q\neq i\) and
\(s=i\), then \(Y=x_ix_q\partial_i-x_q^{(2)}\partial_q\) is divergence-free
and satisfies \([Y,\partial_i]=-x_q\partial_i\).  Finally, if
\(q=i\), then \(s\neq i\), and \(x_i^{(2)}\partial_s\in L_1\) is
divergence-free with
\([x_i^{(2)}\partial_s,\partial_i]=-x_i\partial_s\).  Thus all off-diagonal
generators of \(L_0\) lie in \([L_1,\partial_i]\).

It remains to obtain the diagonal differences.  If \(q,t\neq i\) and \(q\neq t\),
then \(Y=x_ix_q\partial_q-x_ix_t\partial_t\) is divergence-free and
\([Y,\partial_i]=-(x_q\partial_q-x_t\partial_t)\).  If one of the two indices is
\(i\), say \(t\neq i\), then
\(Z=x_ix_t\partial_t-x_i^{(2)}\partial_i\) is divergence-free and
\([Z,\partial_i]=x_i\partial_i-x_t\partial_t\).  Hence all diagonal differences
also lie in \([L_1,\partial_i]\), and the lemma follows.
\end{proof}

\begin{theorem}\label{thm:S}
Let \(\K\) be an algebraically closed field of characteristic \(p>3\), let
\(m\geqslant3\), and let
\(\uln=(n_1,\ldots,n_m)\in\mathbb Z_{>0}^m\).  If
\[
 S(m;\uln)^{(1)}\subseteq G\subseteq S(m;\uln),
\]
then
\[
    \HomLie(G)=\K\id_G .
\]
\end{theorem}

\begin{proof}
Let \(L=S(m;\uln)^{(1)}\) and \(V=L_{-1}\), and take
\(\varphi\in\HomLie(G)\).

\smallskip\noindent\textbf{Preservation of \(V\).}
Fix indices \(j\) and \(b\).
Choose \(i\notin\{j,b\}\), which is possible because \(m\geqslant 3\), and let
\(X=x_i\partial_b\in L_0\).  Then \(X\) is divergence-free,
\([X,\partial_i]=-\partial_b\), and \([\partial_j,X]=0\).  Applying the
Hom--Jacobi identity to \(\partial_i,\partial_j,X\), and using
\([\partial_i,\partial_j]=0\), gives
\([-\partial_b,\varphi(\partial_j)]=0\).
Since \(b\) and \(j\) are arbitrary, \(\varphi(\partial_j)\) centralises all of
\(V\).  By Lemma~\ref{lem:S-centralizer}, \(\varphi(\partial_j)\in V\), and hence
\(\varphi(V)\subseteq V\).

\smallskip\noindent\textbf{Scalar action on \(V\).}
Write
\(\varphi(\partial_j)=\sum_r a_{rj}\partial_r\).  Let \(i\neq j\), and choose
\(b\notin\{i,j\}\).  The element \(X=x_ix_j\partial_b\) is divergence-free and
lies in \(L_1\).  Moreover, \([X,\partial_i]=-x_j\partial_b\) and
\([\partial_j,X]=x_i\partial_b\).  The Hom--Jacobi identity for
\(\partial_i,\partial_j,X\) gives
\[
\begin{split}
    0
    &=
    [-x_j\partial_b,\varphi(\partial_j)]
    +
    [x_i\partial_b,\varphi(\partial_i)]  
    =
    (a_{jj}-a_{ii})\partial_b .
\end{split}
\]
Thus all diagonal coefficients \(a_{ii}\) are equal.

To show that the off-diagonal coefficients vanish, take
\(X=x_i^{(2)}\partial_b\), with
\(b\neq i\).  This element is divergence-free, \([X,\partial_i]=-x_i\partial_b\),
and \([\partial_j,X]=0\).  Hence the Hom--Jacobi identity gives
\(0=[-x_i\partial_b,\varphi(\partial_j)]=a_{ij}\partial_b\).
Therefore \(a_{ij}=0\) whenever \(i\neq j\), and so
\(\varphi|_V=\lambda\id_V\) for some \(\lambda\in\K\).

Let \(\psi=\varphi-\lambda\id_G\).  Then \(\psi\in\HomLie(G)\) and
\(\psi(V)=0\).  We shall prove that \(\psi=0\).

\smallskip\noindent\textbf{Vanishing on \(L_0\).}
It is enough to consider the
standard spanning set consisting of the off-diagonal elements \(x_c\partial_d\),
\(c\neq d\), and the diagonal differences \(x_c\partial_c-x_d\partial_d\),
\(c\neq d\).

Let \(k\) be one of these generators.  We first prove that \(\psi(k)\in V\).
Fix \(b\).  Choose \(i\) such that \([k,\partial_i]=0\): for
\(k=x_c\partial_d\) with \(c\neq d\), take \(i\neq c\); for
\(k=x_c\partial_c-x_d\partial_d\), take \(i\notin\{c,d\}\).  This is possible
because \(m\geqslant 3\).  Choose \(h\in L_0\) with \([\partial_i,h]=\partial_b\):
if \(b\neq i\), take \(h=x_i\partial_b\), while if \(b=i\), take
\(h=x_i\partial_i-x_t\partial_t\) with \(t\neq i\).  Applying the Hom--Jacobi
identity to \(\partial_i,h,k\), and using \(\psi(\partial_i)=0\) and
\([k,\partial_i]=0\), gives \([\partial_b,\psi(k)]=0\).  Since \(b\) is arbitrary,
Lemma~\ref{lem:S-centralizer} gives \(\psi(k)\in V\).

We now prove that this vector is zero.  Choose \(u=\partial_i\in V\) with
\([u,k]=0\), as above.  For \(Y\in L_1\), the Hom--Jacobi identity applied to
\(u,k,Y\) gives
\([[Y,u],\psi(k)]=0\),
because \(\psi(u)=0\) and \([u,k]=0\).  By
Lemma~\ref{lem:S-L1-generates-L0}, the elements \([Y,u]\), with \(Y\in L_1\),
span \(L_0\).  Hence \([L_0,\psi(k)]=0\).  But \(V\) has no non-zero
\(L_0\)-fixed vector.  Indeed, if \(v=\sum_r\mu_r\partial_r\in V\) and
\([x_a\partial_b,v]=0\) for all \(a\neq b\), then
\([x_a\partial_b,v]=-\mu_a\partial_b\) forces every \(\mu_a\) to vanish.  Thus
\(\psi(k)=0\).  Since the chosen generators span \(L_0\), we have
\(\psi(L_0)=0\).

Finally, \([V,L_0]=V\), Lemma~\ref{lem:S-centralizer} gives
\(\Cent_G(V)=V\), and \(L_0=[L_0,L_0]\).  The natural \(L_0\)-module \(V\)
has no non-zero fixed vector, as the off-diagonal elements already show.
Lemma~\ref{lem:propagation}, applied in \(G\) with \(A=L_0\), gives
\(\psi=0\).  Therefore \(\varphi=\lambda\id_G\).
\end{proof}

\subsection{Hamiltonian algebras}

Let \(\K\) be algebraically closed of characteristic \(p>3\), let \(r\geqslant1\),
and let
\(\uln=(n_1,\ldots,n_{2r})\in\mathbb Z_{>0}^{2r}\).  Write
\(\OO=\OO(2r;\uln)\), as defined in Section~\ref{sec:witt}, and write its
variables as symplectic pairs
\(q_1,\ldots,q_r,p_1,\ldots,p_r\).  For \(f,g\in\OO\), define the
\emph{Poisson bracket}
\[
    \{f,g\}
    =
    \sum_{i=1}^r
    \left(
        \partial_{q_i}(f)\partial_{p_i}(g)
        -
        \partial_{p_i}(f)\partial_{q_i}(g)
    \right).
\]
In particular,
\(\{q_i,p_j\}=\delta_{ij}\) and all brackets between two \(q\)'s or two
\(p\)'s vanish.  The \emph{Hamiltonian derivation} associated with \(f\) is
\[
    D_H(f)
    =
    \sum_{i=1}^r
    \left(
        \partial_{q_i}(f)\partial_{p_i}
        -
        \partial_{p_i}(f)\partial_{q_i}
    \right).
\]
It satisfies \(D_H(f)(g)=\{f,g\}\) and
\([D_H(f),D_H(g)]=D_H(\{f,g\})\).  The linear map
\(D_H\colon\OO\to W(2r;\uln)\), \(f\mapsto D_H(f)\), has kernel \(\K\).

To define the ambient algebra, let
\(\omega_H=\sum_{i=1}^r dq_i\wedge dp_i\) be the standard symplectic
differential form.  A special derivation \(D\) acts on divided-power
differential forms by the Lie derivative \(\mathcal L_D\), characterised by
\(\mathcal L_D(f)=D(f)\), \(\mathcal L_D(df)=d(D(f))\), and the graded
Leibniz rule.  The \emph{full Hamiltonian algebra} is
\[
    H(2r;\uln)
    =
    \{D\in W(2r;\uln)\mid \mathcal L_D(\omega_H)=0\}.
\]
Equivalently, this condition gives a purely algebraic definition.  Write
\(D=\sum_{i=1}^r(a_i\partial_{q_i}+b_i\partial_{p_i})\).  Then
\(D\in H(2r;\uln)\) if and only if, for all \(1\leqslant i,j\leqslant r\),
\begin{equation}\label{eq:full-H-coordinate-conditions}
\begin{aligned}
    \partial_{q_j}(b_i)=\partial_{q_i}(b_j),\quad   
    \partial_{p_j}(a_i)=\partial_{p_i}(a_j),\quad
    \partial_{q_j}(a_i)+\partial_{p_i}(b_j) =0.
\end{aligned}
\end{equation}
Equivalently, \eqref{eq:full-H-coordinate-conditions} gives an entirely
algebraic description of the full Hamiltonian algebra.
It is graded by
\(H(2r;\uln)_d=H(2r;\uln)\cap W(2r;\uln)_d\); for a homogeneous Hamiltonian
function \(f\), one has \(\deg D_H(f)=\deg(f)-2\).
Every \(D_H(f)\) belongs to \(H(2r;\uln)\), although the image of \(D_H\)
need not be the full Hamiltonian algebra.  If
\(\tau=(p^{n_1}-1,\ldots,p^{n_{2r}}-1)\), then its simple second derived
algebra is
\[
    L=H(2r;\uln)^{(2)}
     =\Span_{\K}\{D_H(x^{(\alpha)})\mid
          \alpha\neq0,\ \alpha\neq\tau\}.
\]
See
\cite[Chapter~4, Section~4.2, especially (4.2.4)--(4.2.5)]{StradeI}.

We identify \(D_H(f)\) with its Hamiltonian function modulo constants whenever
that function represents an element of \(L\).  Thus an expression such as
\(\varphi(q_i)\) means the image of the Hamiltonian class represented by
\(q_i\).  The standard grading is given by
\(\deg D_H(x^{(\alpha)})=|\alpha|-2\), and its low-degree components are
\[
\begin{aligned}
    V=L_{-1}
      &=\Span_{\K}\{\partial_{q_i},\partial_{p_i}\mid1\leqslant i\leqslant r\}
        \simeq\OO_1,\\
    Q=L_0&\simeq\OO_2\simeq\mathfrak{sp}(V),
    & L_1&\simeq\OO_3 .
\end{aligned}
\]
Here \(\OO_d\) denotes the total-degree-\(d\) component of the divided-power
algebra, and the map \(\ell\mapsto D_H(\ell)\) identifies \(\OO_1\) with
\(V\).  Through this identification, let
\(\omega_V(u,v)=\{u,v\}\) for \(u,v\in V\), and define
\[
    \mathfrak{sp}(V)
    =
    \{A\in\End_{\K}(V)\mid
       \omega_V(Au,v)+\omega_V(u,Av)=0
       \text{ for all }u,v\in V\}.
\]
The displayed isomorphism \(Q\simeq\mathfrak{sp}(V)\) sends a quadratic
Hamiltonian \(q\) to the endomorphism \(u\mapsto\{q,u\}\).

The excluded top monomial has Hamiltonian degree
\(|\tau|-2=\sum_{i=1}^{2r}(p^{n_i}-1)-2\geqslant6\).  Consequently every
Hamiltonian function of total degree \(1\), \(2\), or \(3\) represents an
element of \(H(2r;\uln)^{(2)}\).  Multiplication induces the
usual polarisation identifications between \(\OO_d\) and \(S^d(V)\) for
\(d=1,2,3\): explicitly, the divided-power monomial \(x^{(\alpha)}\) is
identified with \(x^\alpha/\alpha!\), where
\(\alpha!=\alpha_1!\cdots\alpha_{2r}!\) is invertible in \(\K\), since each
\(\alpha_i\leqslant|\alpha|\leqslant3<p\).  We retain divided-power notation throughout.
For later use, write the coordinate vector as
\(\mathbf z=(z_1,\ldots,z_{2r})=(q_1,\ldots,q_r,p_1,\ldots,p_r)\).  If
\(\ell=\sum_{a=1}^{2r}c_a z_a\in V\), its \emph{divided square} is
\[
    \ell^{(2)}
    =
    \sum_{a=1}^{2r}c_a^2z_a^{(2)}
    +
    \sum_{1\leqslant a<b\leqslant2r}c_ac_bz_az_b.
\]
Equivalently, \(\ell^2=2\ell^{(2)}\).  The proof below uses only these
low-degree components.

\begin{lemma}\label{lem:H-centralizer}
Let \(\K\) be an algebraically closed field of characteristic \(p>3\), let
\(r\geqslant1\), and let
\(\uln=(n_1,\ldots,n_{2r})\in\mathbb Z_{>0}^{2r}\).  Let
\(V=\Span_{\K}\{\partial_{q_i},\partial_{p_i}\mid1\leqslant i\leqslant r\}\).
Let \(G\) be a Lie subalgebra satisfying
\(H(2r;\uln)^{(2)}\subseteq G\subseteq H(2r;\uln)\).  Then
\[
    \Cent_G(V)=V .
\]
\end{lemma}

\begin{proof}
Since \(G\subseteq W(2r;\uln)\), write $ D=\sum_{i=1}^r
      \bigl(f_i\partial_{q_i}+g_i\partial_{p_i}\bigr)\in G $.
If \(D\) centralises \(V\), then every divided-power partial derivative of
every \(f_i\) and every \(g_i\) vanishes.  Hence all these coefficients are
scalars and \(D\in V\).  Conversely, \(V\) is abelian and is contained in
\(H(2r;\uln)^{(2)}\subseteq G\).  Thus \(\Cent_G(V)=V\).
\end{proof}

We shall use the following elementary low-degree facts several times.

\begin{lemma}\label{lem:H-low-degree}
Let \(\K\) be an algebraically closed field of characteristic \(p>3\), let
\(r\geqslant1\), and let
\(\uln=(n_1,\ldots,n_{2r})\in\mathbb Z_{>0}^{2r}\).  Let
\(L=H(2r;\uln)^{(2)}\), \(V=L_{-1}\), and \(Q=L_0\), and write
\(\OO_3\) for the total-degree-three component in the Hamiltonian function
realisation.  Using \(D_H\colon\OO_1\to V\), identify \(V\) with the space of linear
Hamiltonians.  If \(0\neq u\in V\), then the linear maps
\[
    Q\longrightarrow V,\quad q\longmapsto\{q,u\},
    \qquad
    \OO_3\longrightarrow Q,\quad f\longmapsto\{f,u\},
\]
are surjective.
Moreover, the divided squares \(\ell^{(2)}\), with \(\ell\in V\), span
\(Q\simeq S^2V\), each satisfies \(\{\ell,\ell^{(2)}\}=0\), and
\([Q,Q]=Q\).
\end{lemma}

\begin{proof}
Choose \(v\in V\) such that \(\{v,u\}=1\).  We first prove that the map from
\(Q\) to \(V\) is surjective.  Given \(w\in V\), let \(c=\{w,u\}\) and
\(w_0=w-cv\).  Then \(\{w_0,u\}=0\).  Therefore
\[
    \{vw_0+c\,v^{(2)},u\}=w_0+cv=w .
\]
Thus every \(w\in V\) lies in the image.

To prove the second surjectivity assertion, let
\(V_0=\ker\{\,\cdot\,,u\}\), so that \(V=\K v\oplus V_0\).  Every
\(q\in Q\simeq S^2V\) has a unique expression
\[
    q=q_0+vw+c\,v^{(2)}
    \qquad(q_0\in S^2V_0,\ w\in V_0,\ c\in\K).
\]
Since \(\{v,u\}=1\) and \(\{V_0,u\}=0\), the cubic
\(vq_0+v^{(2)}w+c\,v^{(3)}\) satisfies
\[
    \{vq_0+v^{(2)}w+c\,v^{(3)},u\}=q.
\]
Thus \(\{\,\cdot\,,u\}\) maps \(\OO_3\) onto \(Q\).

Since \(p\neq2\), the quadratic component \(Q\simeq S^2V\) is spanned by the
divided squares \(\ell^{(2)}\), and
\(\{\ell,\ell^{(2)}\}=0\).  The final assertion follows from the identification
\(Q\simeq\mathfrak{sp}(V)\), since \(\mathfrak{sp}(V)\) is perfect for \(p>3\).
\end{proof}

\begin{theorem}\label{thm:H}
Let \(\K\) be an algebraically closed field of characteristic \(p>3\), let
\(r\geqslant1\), and let
\(\uln=(n_1,\ldots,n_{2r})\in\mathbb Z_{>0}^{2r}\).  If
\[
 H(2r;\uln)^{(2)}\subseteq G\subseteq H(2r;\uln),
\]
then
\[
 \HomLie(G)=\K\id_G.
\]
\end{theorem}

\begin{proof}
Let \(L=H(2r;\uln)^{(2)}\), \(V=L_{-1}\), and \(Q=L_0\), and take
\(\varphi\in\HomLie(G)\).

\smallskip\noindent\textbf{Preservation of \(V\).}
Fix \(i\), and write
\[
 \varphi(q_i)=
 \sum_{j=1}^r
 \bigl(a_j\partial_{q_j}+b_j\partial_{p_j}\bigr).
\]
If \(f\in\OO\) is independent of \(p_i\), the Hom--Jacobi identity for the
Hamiltonian derivations represented by \(p_i,q_i,f\) gives
\[
 [D_H(\partial_{q_i}f),\varphi(q_i)]=0.
\]
Taking \(f=q_i^{(2)},q_iq_j,q_ip_j\), with \(j\neq i\), shows that every
coefficient of \(\varphi(q_i)\) depends only on \(q_i\).  The coordinate
conditions \eqref{eq:full-H-coordinate-conditions} then give
\[
 \varphi(q_i)=v_i+B_i(q_i)\partial_{p_i}
 \qquad\text{for some }v_i\in V.
\]
Interchanging \(q_i\) and \(p_i\), we similarly obtain
\[
 \varphi(p_i)=w_i+A_i(p_i)\partial_{q_i}
 \qquad\text{for some }w_i\in V.
\]
The Hom--Jacobi identity for \(q_i,p_i,q_ip_i\) yields
\[
 -B_i'(q_i)\partial_{p_i}+A_i'(p_i)\partial_{q_i}=0.
\]
The kernel of either one-variable divided-power derivative is \(\K\).
Consequently \(A_i\) and \(B_i\) are constant, and
\(\varphi(V)\subseteq V\).

\smallskip\noindent\textbf{Scalar action on \(V\).}
Let \(A=\varphi|_V\).  For \(a,b\in V\) and \(f\in\OO_3\simeq L_1\), the
Hom--Jacobi identity gives
\begin{equation}\label{eq:H-cubic}
 \{\{f,a\},Ab\}=\{\{f,b\},Aa\}.
\end{equation}
For \(f\in\OO_3\), let
\[
 C_f(x,y,z)=\{\{\{f,x\},y\},z\}
 \qquad(x,y,z\in V).
\]
This is a symmetric trilinear form.  Choose linear coordinates
\(y_1,\ldots,y_{2r}\) and the ordered dual tuple
\(\mathbf y^\vee=(y_1^\vee,\ldots,y_{2r}^\vee)\), where
\(\{y_i,y_j^\vee\}=\delta_{ij}\).  If
\(f=\sum_{|\alpha|=3}c_\alpha y^{(\alpha)}\), then the values of \(C_f\) on
triples of elements of \(\mathbf y^\vee\) recover the coefficients \(c_\alpha\).
Thus \(f\mapsto C_f\) is injective.  Since
\[
 \dim\OO_3=\binom{2r+2}{3}=\dim S^3(V^*),
\]
it is an isomorphism from \(\OO_3\) onto \(S^3(V^*)\).

Since both sides of \eqref{eq:H-cubic} belong to \(\OO_1\) and \(D_H\) is injective on \(\OO_1\), this is an equality of linear Hamiltonian fuctions. Taking the Poisson bracket with \(t\in V\), we obtain \(C_f(a,Ab,t)=C_f(b,Aa,t)\). Then using the non-degenerate
pairing between \(S^3(V^*)\) and \(S^3V\), we obtain
\[
 (a\odot Ab-b\odot Aa)\odot t=0.
\]
Multiplication by a non-zero linear form is injective in \(S(V)\), so
\(a\odot Ab=b\odot Aa\).  If \(a,b\) are linearly independent, write
\[
 Aa=xa+yb+w,\qquad Ab=ra+sb+z,
\]
where \(w,z\) lie in a fixed complement of \(\Span_{\K}\{a,b\}\).
Comparison of components in the resulting direct-sum decomposition of
\(S^2V\) gives \(r=y=0\), \(w=z=0\), and \(s=x\).  Hence every
one-dimensional subspace of \(V\) is \(A\)-stable, and
\(A=\lambda\id_V\) for some \(\lambda\in\K\).

Let \(\psi=\varphi-\lambda\id_G\).  Then
\(\psi\in\HomLie(G)\) and \(\psi(V)=0\).

\smallskip\noindent\textbf{Vanishing on \(Q\).}
Fix \(0\neq\ell\in V\) and let \(q=\ell^{(2)}\).  Since
\([\ell,q]=0\), the Hom--Jacobi identity for \(\ell,q,k\), with \(k\in Q\),
gives
\[
 [[k,\ell],\psi(q)]=0.
\]
Lemma~\ref{lem:H-low-degree} shows that \([Q,\ell]=V\).  Hence
\(\psi(q)\in\Cent_G(V)=V\), by Lemma~\ref{lem:H-centralizer}.  Applying the
Hom--Jacobi identity to \(\ell,q,C\), with \(C\in L_1\), gives
\[
 [[C,\ell],\psi(q)]=0.
\]
Again by Lemma~\ref{lem:H-low-degree}, \([L_1,\ell]=Q\).  The natural
\(Q\simeq\mathfrak{sp}(V)\)-module \(V\) has no non-zero fixed vector, so
\(\psi(q)=0\).  The divided squares span \(Q\); hence \(\psi(Q)=0\).

Finally, \([V,Q]=V\), \(\Cent_G(V)=V\), and \([Q,Q]=Q\).  Since \(V\) has no
non-zero \(Q\)-fixed vector, Lemma~\ref{lem:propagation}, applied in \(G\)
with \(A=Q\), gives \(\psi=0\).  Therefore
\(\varphi=\lambda\id_G\).
\end{proof}

\subsection{Contact algebras}

Let \(\K\) be algebraically closed of characteristic \(p>3\), let \(r\geqslant1\),
and let
\(\uln=(n_1,\ldots,n_{2r+1})\in\mathbb Z_{>0}^{2r+1}\).  Write
\(\OO=\OO(2r+1;\uln)\).  We choose symplectic coordinates
\(q_1,\ldots,q_r,p_1,\ldots,p_r\) and a contact variable \(t\).  Let \(V\) be the
\(2r\)-dimensional vector space spanned by the linear symplectic variables, and
let \(\omega\) be the symplectic form on \(V\), normalised by
\[
    \omega(q_i,p_j)=\delta_{ij},
    \qquad
    \omega(q_i,q_j)=\omega(p_i,p_j)=0 .
\]

For \(f\in\OO\), define its \emph{contact derivation} by
\[
\begin{split}
    D_K(f)
    ={}&
    \sum_{i=1}^r
    \bigl(q_i\partial_t(f)-\partial_{p_i}(f)\bigr)\partial_{q_i}
    +
    \sum_{i=1}^r
    \bigl(p_i\partial_t(f)+\partial_{q_i}(f)\bigr)\partial_{p_i} \\
    &+
    \left(2f-
        \sum_{i=1}^r
        \bigl(q_i\partial_{q_i}(f)+p_i\partial_{p_i}(f)\bigr)
    \right)\partial_t .
\end{split}
\]
The map \(f\mapsto D_K(f)\) is injective, and the \emph{full contact algebra} is
\[
    K(2r+1;\uln)=\{D_K(f)\mid f\in\OO\}\subseteq W(2r+1;\uln).
\]
With our convention, its bracket is given by the \emph{contact bracket} of
Hamiltonians
\begin{equation}\label{eq:contact-bracket}
\begin{split}
    \langle f,g\rangle
    &=D_K(f)(g)-2g\partial_t(f)
    =\Delta(f)\partial_t(g)-\Delta(g)\partial_t(f)+\{f,g\}_0.
\end{split}
\end{equation}
Here \(\{-,-\}_0\) is the symplectic Poisson bracket in the variables
\(q_i,p_i\), namely
\[
    \{f,g\}_0
    =
    \sum_{i=1}^r
    \left(
        \partial_{q_i}(f)\partial_{p_i}(g)
        -
        \partial_{p_i}(f)\partial_{q_i}(g)
    \right),
\]
and
\[
    \Delta(f)
    =
    2f-\sum_{i=1}^r
    \bigl(q_i\partial_{q_i}(f)+p_i\partial_{p_i}(f)\bigr).
\]
Thus \([D_K(f),D_K(g)]=D_K(\langle f,g\rangle)\).  The simple contact algebra
used below is the first derived algebra \(L=K(2r+1;\uln)^{(1)}\).
If \(p\nmid(2r+4)\), then \(L=K(2r+1;\uln)\); if \(p\mid(2r+4)\), then \(L\)
has codimension one and is spanned by the \(D_K(x^{(\alpha)})\) with
\(\alpha\neq(p^{n_1}-1,\ldots,p^{n_{2r+1}}-1)\).  See
\cite[Chapter~4, Section~4.2, especially (4.2.7)--(4.2.10) and
(4.2.12)--(4.2.16)]{StradeI}.  In the calculations below we identify
\(D_K(f)\) with its Hamiltonian \(f\).

The standard contact grading assigns degree \(1\) to the variables in \(V\) and
degree \(2\) to \(t\), and then subtracts \(2\) on passing from a Hamiltonian
to its contact derivation.  Explicitly,
\[
    \deg D_K(x^{(\alpha)})
    =
    \alpha_1+\cdots+\alpha_{2r}+2\alpha_{2r+1}-2 .
\]
Thus
\[
    L=L_{-2}\oplus L_{-1}\oplus L_0\oplus L_1\oplus\cdots,
\]
where
\[
\begin{aligned}
    L_{-2}&=\K c,\qquad c=1,\\
    L_{-1}&=V,\\
    L_0&=\Span_{\K}\{uv\mid u,v\in V\}\oplus\K t
       \simeq \mathfrak{sp}(V)\oplus \K E,
       \qquad E=t .
\end{aligned}
\]
The quadratic Hamiltonians in the variables from \(V\), independent of \(t\),
form the displayed copy of \(\mathfrak{sp}(V)\): a quadratic Hamiltonian
\(f\) corresponds to the endomorphism \(u\mapsto\{f,u\}_0\) of \(V\).
Here
\[
    \mathfrak{sp}(V)
    =
    \{A\in\End_{\K}(V)\mid
      \omega(Au,v)+\omega(u,Av)=0
      \text{ for all }u,v\in V\}.
\]

If \(p\mid(2r+4)\), the only Hamiltonian omitted from
\(K(2r+1;\uln)^{(1)}\) is the top divided-power monomial.  Its contact degree is
\[
    \sum_{i=1}^{2r}(p^{n_i}-1)
    +2(p^{n_{2r+1}}-1)-2
    \geqslant14.
\]
Hence every Hamiltonian of contact degree at most \(2\) used below lies in the
derived algebra.  If \(p\nmid(2r+4)\), the same assertion is immediate because
the full contact algebra is already equal to its derived algebra.  Direct
substitution in \eqref{eq:contact-bracket} now gives the following low-degree
relations.  Here \(t^{(2)}\) is the second divided power of the contact
variable, characterised by \(t^2=2t^{(2)}\):
\begin{align}
    [c,u]&=0,\qquad [u,v]=\omega(u,v)c,
        && u,v\in V, \label{eq:K-cuv}\\
    [A,c]&=0,
        && A\in\mathfrak{sp}(V), \notag\\
    [E,c]&=-2c,\qquad [E,u]=-u,
        && u\in V, \notag\\
    [c,tv]&=2v,
        && v\in V, \label{eq:K-ctv}\\
    [u,tv]-\omega(u,v)E&\in\mathfrak{sp}(V),
        && u,v\in V, \label{eq:K-utv}\\
    [c,t^{(2)}]&=2E. \label{eq:K-ct2}
\end{align}

\begin{lemma}\label{lem:K-centralizer}
Let \(\K\) be an algebraically closed field of characteristic \(p>3\), let
\(r\geqslant1\), and let
\(\uln=(n_1,\ldots,n_{2r+1})\in\mathbb Z_{>0}^{2r+1}\).  Let
\(V=K(2r+1;\uln)^{(1)}_{-1}\) and \(c=D_K(1)\), represented by the contact
Hamiltonian \(1\) under the fixed identification; then
\(K(2r+1;\uln)^{(1)}_{-2}=\K c\).  Let \(G\) be a Lie subalgebra satisfying
\(K(2r+1;\uln)^{(1)}\subseteq G\subseteq K(2r+1;\uln)\).  Then
\[
    \Cent_G(V)=\K c .
\]
\end{lemma}

\begin{proof}
Let \(z\in G\) centralise \(V\), and represent \(z\) by a contact Hamiltonian
\(f\).  For \(u\in V\), the condition \([f,u]=0\) is, by
\eqref{eq:contact-bracket},
\[
    \{f,u\}_0-u\partial_t(f)=0 .
\]
For \(u\in V\), let
\(X_u=\{\,\cdot\,,u\}_0-u\partial_t\).  Thus \(X_u(f)=0\) for every
\(u\in V\).  For \(u,v\in V\), a direct calculation gives
\[
    [X_u,X_v]=2\omega(u,v)\partial_t .
\]
Since \(p>3\), the scalar \(2\) is invertible.  Choosing \(u,v\) with
\(\omega(u,v)\neq0\), we obtain \(\partial_t(f)=0\).  The equations
\(X_u(f)=0\) then reduce to \(\{f,u\}_0=0\) for all \(u\in V\).  Hence all
symplectic first derivatives of \(f\) vanish.  In the divided-power algebra this
forces \(f\in\K\).  Therefore \(z\in\K c\).

Conversely, \([c,u]=0\) for every \(u\in V\) by \eqref{eq:K-cuv}.  Hence
\(\Cent_G(V)=\K c\).
\end{proof}

\begin{theorem}\label{thm:K}
Let \(\K\) be an algebraically closed field of characteristic \(p>3\), let
\(r\geqslant1\), and let
\(\uln=(n_1,\ldots,n_{2r+1})\in\mathbb Z_{>0}^{2r+1}\).  If $K(2r+1;\uln)^{(1)}\subseteq G\subseteq K(2r+1;\uln)$, 
then
\[
    \HomLie(G)=\K\id_G .
\]
\end{theorem}

\begin{proof}
Let \(L=K(2r+1;\uln)^{(1)}\), and take \(\varphi\in\HomLie(G)\).

\smallskip\noindent\textbf{The component \(L_{-2}\).}
Let
\(A\in\mathfrak{sp}(V)\) and \(u\in V\).  Since \([c,u]=0\) and \([A,c]=0\), the
Hom--Jacobi identity for \(c,u,A\) gives
\([[u,A],\varphi(c)]=0\).
As \(u\) and \(A\) vary, the elements \([u,A]\) span the natural
\(\mathfrak{sp}(V)\)-module \(V\).  Hence \([V,\varphi(c)]=0\).  By
Lemma~\ref{lem:K-centralizer}, \(\varphi(c)\in\K c\).  Write
\(\varphi(c)=\lambda c\), and let \(\psi=\varphi-\lambda\id_G\).  Then
\(\psi\in\HomLie(G)\) and \(\psi(c)=0\).

\smallskip\noindent\textbf{The component \(L_{-1}\).}
Take \(u,v\in V\), and let \(Y=tv\in L_1\).  Applying the Hom--Jacobi identity
to \(c,u,Y\), and using
\([c,u]=0\) and \(\psi(c)=0\), gives
\([[Y,c],\psi(u)]=0\).
By \eqref{eq:K-ctv}, \([Y,c]=-2v\).  Thus \([v,\psi(u)]=0\) for all \(v\in V\).
Lemma~\ref{lem:K-centralizer} implies \(\psi(u)\in\K c\).  Hence
\(\psi(u)=\ell(u)c\) for some linear functional \(\ell\in V^*\).

Apply the Hom--Jacobi identity to \(c,u,t^{(2)}\).
The relations \([c,u]=0\) and \(\psi(c)=0\) give
\[
    [[t^{(2)},c],\psi(u)]=0.
\]
By \eqref{eq:K-ct2}, \([E,\ell(u)c]=0\).  Since \([E,c]=-2c\) and \(p>3\),
we have \(\ell(u)=0\).  Thus \(\psi(V)=0\).

\smallskip\noindent\textbf{The component \(L_0\).}
First let
\(A\in\mathfrak{sp}(V)\).  Applying the Hom--Jacobi identity to \(c,A,tv\), and
using \([c,A]=0\) and \(\psi(c)=0\), gives
\([[tv,c],\psi(A)]=0\).
As \(v\) varies, the elements \([tv,c]\) span \(V\).  Hence
\([V,\psi(A)]=0\), and Lemma~\ref{lem:K-centralizer} gives
\(\psi(A)\in\K c\).  Applying the Hom--Jacobi identity to \(c,A,t^{(2)}\), we
obtain
\([[t^{(2)},c],\psi(A)]=0\).
Since \([t^{(2)},c]\) is a non-zero scalar multiple of \(E\), and \(E\) acts
by a non-zero scalar on \(\K c\), it follows that \(\psi(A)=0\).  Therefore
\(\psi(\mathfrak{sp}(V))=0\).

It remains to deal with the grading element \(E\).  Applying the Hom--Jacobi
identity to \(u,A,E\), with \(u\in V\) and \(A\in\mathfrak{sp}(V)\), gives
\([[u,A],\psi(E)]=0\),
because \(\psi(u)=0\) and \(\psi(A)=0\).  Since the elements \([u,A]\) span
\(V\), Lemma~\ref{lem:K-centralizer} gives \(\psi(E)=\alpha c\) for some
\(\alpha\in\K\).

We show that \(\alpha=0\).  Before doing so, note that the same Hom--Jacobi
argument with \(u,A,Y\) in place of \(u,A,E\) shows that \(\psi(Y)\in\K c\) for
every \(Y\in G\), and in particular for every \(Y\in L_1\).  Choose
\(u,v\in V\) with \(\omega(u,v)\neq0\), and let \(Y=tv\).  Applying the
Hom--Jacobi identity to \(u,Y,E\), we obtain
\[
    [[u,Y],\psi(E)]+[[E,u],\psi(Y)]=0,
\]
because \(\psi(u)=0\).  The second term is zero: \([E,u]\in V\),
\(\psi(Y)\in\K c\), and \([V,c]=0\).  Hence
\([[u,Y],\alpha c]=0\).
By \eqref{eq:K-utv}, the difference
\([u,Y]-\omega(u,v)E\) lies in \(\mathfrak{sp}(V)\) and hence annihilates
\(c\), whereas \([E,c]=-2c\).  Since \(\omega(u,v)\neq0\), we must have
\(\alpha=0\).  Thus \(\psi(E)=0\), and hence \(\psi(L_0)=0\).

\smallskip\noindent\textbf{Conclusion.}
The preceding argument shows that there is a linear functional
\(\beta\colon G\to\K\) such that \(\psi(X)=\beta(X)c\) for \(X\in G\), and that
\(\beta(L_{-2}\oplus L_{-1}\oplus L_0)=0\).
Fix \(X\in G\).  Consider the linear functional \(v\mapsto\beta(tv)\) on \(V\).
Since \(\dim V=2r\geqslant2\), we may choose a non-zero \(v\in V\) with
\(\beta(tv)=0\).  By non-degeneracy of \(\omega\), choose \(u\in V\) such that
\(\omega(u,v)\neq0\), and let \(Y=tv\).  Then \(\psi(u)=0\) and \(\psi(Y)=0\).
The Hom--Jacobi identity for \(u,Y,X\) gives
\([[u,Y],\psi(X)]=0\).
Therefore \(\beta(X)[[u,Y],c]=0\).  By \eqref{eq:K-utv} and
\([\mathfrak{sp}(V),c]=0\),
one has \([[u,Y],c]=\omega(u,v)[E,c]\neq0\).
Thus \(\beta(X)=0\).  Since \(X\) was arbitrary, \(\beta=0\), so \(\psi=0\).
Consequently \(\varphi=\lambda\id_G\).
\end{proof}

\section{The Melikian algebras}\label{sec:melikian}

In this section \(\K\) is algebraically closed of characteristic \(5\).
Let \(\uln=(n_1,n_2)\in\mathbb Z_{>0}^2\).  We write
\(M=M(2;\uln)\) for the Melikian algebra usually denoted \(M(n_1,n_2)\).
We use Strade's standard realisation and record its bracket explicitly.

\subsection{The standard realisation and grading}

As a vector space,
\[
    M=W(2;\uln)\oplus \OO(2;\uln)\oplus \widetilde W(2;\uln),
\]
where \(\widetilde W(2;\uln)\) denotes a second copy of \(W(2;\uln)\).  We shall
write elements of the three summands as \(D\in W(2;\uln)\),
\(f\in \OO(2;\uln)\), and \(\widetilde F\in\widetilde W(2;\uln)\).
If \(D,E\in W(2;\uln)\), \(f,g\in\OO(2;\uln)\), and
\(F=f_1\partial_1+f_2\partial_2\) and
\(G=g_1\partial_1+g_2\partial_2\),
then the bracket is given by
\begin{align}
    [D,E]&=[D,E]_{W}, \label{eq:M-WW}\\
    [D,f]&=D(f)-2\diver(D)f, \label{eq:M-Df}\\
    [D,\widetilde F]&=\widetilde{[D,F]_{W}}+2\diver(D)\widetilde F, \notag\\
    [f,\widetilde F]&=fF, \label{eq:M-fE}\\
    [f,g]&=
        2\bigl(g\partial_2(f)-f\partial_2(g)\bigr)\widetilde\partial_1
        \notag\\
        &\qquad
        +2\bigl(f\partial_1(g)-g\partial_1(f)\bigr)\widetilde\partial_2,
        \notag\\
    [\widetilde F,\widetilde G]&=f_1g_2-f_2g_1 . \label{eq:M-tilde}
\end{align}
Here \([D,E]_W\) is the Witt bracket and
\(\diver(f_1\partial_1+f_2\partial_2)
=\partial_1(f_1)+\partial_2(f_2)\).  The tilde denotes the fixed vector-space
isomorphism from \(W(2;\uln)\) to its second copy.  In
\([f,\widetilde F]=fF\), the right-hand side is
\((ff_1)\partial_1+(ff_2)\partial_2\) in the first Witt summand, whereas
\([\widetilde F,\widetilde G]=f_1g_2-f_2g_1\) belongs to the middle summand.
We also write \(f\widetilde\partial_i=\widetilde{f\partial_i}\).  Together
with bilinearity and skew-symmetry, these formulas define a Lie algebra.
The resulting algebra \(M(2;\uln)\) is simple
\cite[(4.3.1) and Theorem~4.3.3(1)]{StradeI}.  In particular,
\(\dim_{\K}M=5^{n_1+n_2+1}\).

Write \(\OO(2;\uln)=\bigoplus_{d\geqslant0}\OO_d\) and
\(W(2;\uln)=\bigoplus_{k\geqslant-1}W_k\) for the total-degree and Witt gradings
defined above.  The \emph{canonical Melikian grading} is determined on
homogeneous elements by
\[
\begin{aligned}
    \deg_M(f\partial_i)&=3\deg(f)-3,\\
    \deg_M(f)&=3\deg(f)-2,\\
    \deg_M(\widetilde{f\partial_i})&=3\deg(f)-1.
\end{aligned}
\]
Equivalently,
\[
    M_{3k}=W_k,
    \qquad
    M_{3k+1}=\OO_{k+1},
    \qquad
    M_{3k+2}=\widetilde W_k,
\]
where \(\widetilde W_k=\{\widetilde D\mid D\in W_k\}\), and a component outside
the range of the indicated grading is zero.  It
has depth \(3\), and its low-degree part is
\[
\begin{aligned}
    M_{-3}&=U=\Span_{\K}\{d_1,d_2\}, & \mathbf d&=(d_1,d_2),\quad d_i=\partial_i,\\
    M_{-2}&=\K c, & c&=1,\\
    M_{-1}&=T=\Span_{\K}\{e_1,e_2\}, & \mathbf e&=(e_1,e_2),\quad e_i=\widetilde\partial_i,\\
    M_0&=\Span_{\K}\{a_{ij}=x_i\partial_j\mid1\leqslant i,j\leqslant2\}
        \simeq\mathfrak{gl}_2,\\
    M_1&=\Span_{\K}\{x_1,x_2\}, & \mathbf x&=(x_1,x_2).
\end{aligned}
\]
The boldface tuples \(\mathbf d\), \(\mathbf e\), and \(\mathbf x\) record the
ordered bases of the corresponding two-dimensional components.  Here
\(d_i,a_{ij}\) belong to the first Witt summand, \(c,x_i\) to the middle
divided-power summand, and \(e_i\) to the second Witt summand.  Throughout the
relations below, \(i,j,k\in\{1,2\}\).
Write \(M_-=U\oplus\K c\oplus T\).
The low-degree brackets needed below are the following:
\begin{align}
    [e_1,e_2]&=c, \notag\\
    [c,e_i]&=d_i, \notag\\
    [d_i,M_-]&=0, \label{eq:M-low3}\\
    [x_i,d_j]&=-\delta_{ij}c, \label{eq:M-low4}\\
    [x_1,c]&=-2e_2,
    &
    [x_2,c]&=2e_1, \notag\\
    [x_i,e_j]&=a_{ij}, \label{eq:M-low6}\\
    [a_{ij},d_k]&=-\delta_{ik}d_j, \label{eq:M-low7}\\
    [a_{ij},e_k]&=-\delta_{ik}e_j+2\delta_{ij}e_k, \label{eq:M-low8}\\
    [a_{ij},c]&=-2\delta_{ij}c. \label{eq:M-low9}
\end{align}
These identities follow directly from \eqref{eq:M-WW}--\eqref{eq:M-tilde}.
We shall also use the degree-two elements \(x_j\widetilde\partial_i\in M_2\),
for which
\begin{equation}\label{eq:M-degree2}
    [d_k,x_j\widetilde\partial_i]=\delta_{kj}e_i .
\end{equation}

\subsection{Centraliser computations}

We first prove two centraliser statements.  They are the only places where the
explicit realisation is used beyond the low-degree bracket table.

\begin{lemma}\label{lem:M-cent1}
Let \(\K\) be an algebraically closed field of characteristic \(5\), let
\(\uln=(n_1,n_2)\in\mathbb Z_{>0}^2\), and let
\(M=M(2;\uln)\).  In the canonical grading, write
\(U=M_{-3}=\Span_{\K}\{d_1,d_2\}\), \(M_{-2}=\K c\),
\(T=M_{-1}=\Span_{\K}\{e_1,e_2\}\), and
\(M_1=\Span_{\K}\{x_1,x_2\}\).  If \(0\neq u\in U\), then there exists
\(x\in M_1\) such that \([x,u]=0\) and
\[
    \Cent_M\bigl(U+\K c+\K[c,x]+[T,x]\bigr)=\K u .
\]
Here \([T,x]=\{[e,x]\mid e\in T\}\subseteq M_0\).
\end{lemma}

\begin{proof}
Write \(u=\alpha d_1+\beta d_2\), with \((\alpha,\beta)\neq(0,0)\), and let
\(x=-\beta x_1+\alpha x_2\).  By \eqref{eq:M-low4}, \([x,u]=0\).  Moreover,
\([c,x]=-2(\alpha e_1+\beta e_2)\) is non-zero and
\([[c,x],c]=2u\).

Let \(z\in M\) centralise both \(c\) and \(U\).  Write
\[
    z=D+h+\widetilde F,
    \qquad
    D=g_1\partial_1+g_2\partial_2,
    \qquad
    F=f_1\partial_1+f_2\partial_2,
\]
with \(h\in\OO(2;\uln)\).  Using
\eqref{eq:M-Df}--\eqref{eq:M-fE}, we compute
\[
    [c,z]
    =
    F+2\diver(D)c-2\partial_2(h)\widetilde\partial_1
      +2\partial_1(h)\widetilde\partial_2 .
\]
The three summands \(W(2;\uln)\), \(\OO(2;\uln)\), and
\(\widetilde W(2;\uln)\) are independent.  Hence \([c,z]=0\) implies
\[
    F=0,\qquad
    \diver(D)=0,\qquad
    \partial_1(h)=\partial_2(h)=0.
\]
Thus \(h\in\K c\).  Since \(z\) also centralises \(d_1\) and \(d_2\), we have
\[
    [d_i,D]=\sum_s \partial_i(g_s)\partial_s=0
    \qquad (i=1,2).
\]
Therefore the coefficients \(g_s\) are scalar, so \(D\in U\).  We have shown
that any element centralising \(c\) and \(U\) lies in \(U\oplus\K c\).

Now suppose in addition that \(z\) centralises \([c,x]\).  Write
\(z=w+\mu c\), with \(w\in U\).  Since \([U,M_-]=0\), whereas
\([[c,x],c]=2u\neq0\), the equality \([[c,x],z]=0\) forces \(\mu=0\).  Hence
\(z\in U\).

It remains to impose the centralising condition with \([T,x]\).  Directly from
\eqref{eq:M-low6}, one obtains
\[
    [e_1,x]=\beta a_{11}-\alpha a_{21},
    \qquad
    [e_2,x]=\beta a_{12}-\alpha a_{22}.
\]
Write \(z=\rho d_1+\sigma d_2\in U\).  Using \eqref{eq:M-low7}, we get
\[
    [[e_1,x],z]=(-\beta\rho+\alpha\sigma)d_1,
    \qquad
    [[e_2,x],z]=(-\beta\rho+\alpha\sigma)d_2 .
\]
Thus \(z\) centralises \([T,x]\) if and only if
\(\alpha\sigma=\beta\rho\), which is equivalent to \(z\in\K u\).  This proves the
lemma.
\end{proof}

\begin{lemma}\label{lem:M-cent2}
Let \(\K\) be an algebraically closed field of characteristic \(5\), let
\(\uln=(n_1,n_2)\in\mathbb Z_{>0}^2\), and let
\(M=M(2;\uln)\).  In its canonical grading, let
\(U=M_{-3}\), \(M_{-2}=\K c\), \(T=M_{-1}\), and
\(M_-=U\oplus\K c\oplus T\).  Then
\[
    \Cent_M(M_-)=U .
\]
\end{lemma}

\begin{proof}
Let \(z\in M\) centralise \(M_-\).  In particular, \(z\) centralises \(c\) and
\(U\).  The computation in the proof of Lemma~\ref{lem:M-cent1} gives
\(z\in U\oplus\K c\).  Write \(z=w+\mu c\), with \(w\in U\).  Since
\([U,T]=0\) and \([c,e_i]=d_i\), the condition \([z,T]=0\) forces
\(\mu=0\).  Hence \(z\in U\).

Conversely, \eqref{eq:M-low3} says that every element of \(U\) centralises
\(M_-\).  Therefore \(\Cent_M(M_-)=U\).
\end{proof}

\subsection{Hom--Lie structures on the Melikian algebras}

We now prove that every Hom--Lie structure is scalar.  The first step is to
determine a Hom--Lie
structure on the lowest component \(U=M_{-3}\).

\begin{lemma}\label{lem:M1}
Let \(\K\) be an algebraically closed field of characteristic \(5\), let
\(\uln=(n_1,n_2)\in\mathbb Z_{>0}^2\), let
\(M=M(2;\uln)\), and write \(U=M_{-3}\) for the lowest component of the
canonical grading.  If \(\theta\in\HomLie(M)\), then
\(\theta(U)\subseteq U\), and there exists
\(\lambda\in\K\) such that
\[
    \theta|_U=\lambda\id_U .
\]
\end{lemma}

\begin{proof}
Fix \(0\neq u\in U\).  We shall prove that \(\theta(u)\in\K u\).

Let \(v\in U\) and \(x\in M_1\).  Since \([u,v]=0\), the Hom--Jacobi identity for
\(u,v,x\) gives
\[
    [[x,u],\theta(v)]+[[v,x],\theta(u)]=0 .
\]
Choose \(x\in M_1\) with \([x,u]=0\), as in Lemma~\ref{lem:M-cent1}.  Since
\(x\neq0\), we may choose \(v\in U\) with \([x,v]\neq0\).  Then \([x,v]\) is a
non-zero scalar multiple of \(c\), and hence the preceding identity gives
\([c,\theta(u)]=0\).

Next take \(e\in T\), and apply the Hom--Jacobi identity to \(u,e,c\).  Since
\([u,e]=0\) and \([c,u]=0\), we obtain
\([[e,c],\theta(u)]=0\).
As \(e\) varies in \(T\), the elements \([e,c]\) span \(U\).  Thus
\([U,\theta(u)]=0\).

Now choose \(x\in M_1\) as in Lemma~\ref{lem:M-cent1}, so that \([x,u]=0\).
The Hom--Jacobi identity for \(u,c,x\) gives
\([[c,x],\theta(u)]=0\),
and the Hom--Jacobi identity for \(u,e,x\), with \(e\in T\), gives
\([[e,x],\theta(u)]=0\).
Hence \(\theta(u)\) centralises \(U+\K c+\K[c,x]+[T,x]\).  By
Lemma~\ref{lem:M-cent1}, \(\theta(u)\in\K u\).

Since this holds for every non-zero \(u\in U\), the map \(\theta\) preserves every
one-dimensional subspace of \(U\).  In particular, \(\theta(U)\subseteq U\).  If
\(u_1,u_2\) are linearly independent and
\(\theta(u_i)=\lambda_i u_i\), then applying the same conclusion to \(u_1+u_2\)
gives \(\lambda_1=\lambda_2\).  Therefore
\(\theta|_U=\lambda\id_U\) for some \(\lambda\in\K\).
\end{proof}

\begin{lemma}\label{lem:M2}
Let \(\K\) be an algebraically closed field of characteristic \(5\), let
\(\uln=(n_1,n_2)\in\mathbb Z_{>0}^2\), let
\(M=M(2;\uln)\), and write \(U=M_{-3}\), \(M_{-2}=\K c\), and
\(T=M_{-1}\) for the lowest components of its canonical grading.  If
\(\theta\in\HomLie(M)\) and \(\theta(U)=0\), then
\[
    \theta(c)=0
    \qquad\text{and}\qquad
    \theta(T)=0 .
\]
\end{lemma}

\begin{proof}
We use the \(M\)-module stability of \(\HomLie(M)\) from
Lemma~\ref{lem:HomLie-module}.  Let \(0\neq x\in M_1\).  Then
\(\theta\bullet x\in\HomLie(M)\), and for \(u\in U\) we have
\[
    (\theta\bullet x)(u)
    =
    [\theta(u),x]-\theta([u,x])
    =
    -\theta([u,x]).
\]
Since \([U,x]=\K c\), the restriction of \(\theta\bullet x\) to \(U\) has rank at
most \(1\).  By Lemma~\ref{lem:M1}, this restriction must be a scalar multiple of
\(\id_U\).  A non-zero scalar multiple of \(\id_U\) has rank \(2\), whereas here
the rank is at most \(1\).  Hence
\((\theta\bullet x)(U)=0\).
Thus \(\theta([U,x])=0\), and since \([U,x]=\K c\), we get
\(\theta(c)=0\).

Now fix \(i\in\{1,2\}\), and choose any \(j\in\{1,2\}\).  Let
\(y=x_j\widetilde\partial_i\in M_2\).  Again
\(\theta\bullet y\in\HomLie(M)\).  By \eqref{eq:M-degree2}, for \(d_k\in U\),
\[
    (\theta\bullet y)(d_k)
    =
    [\theta(d_k),y]-\theta([d_k,y])
    =
    -\delta_{kj}\theta(e_i).
\]
This linear map from \(U\) has a non-zero kernel.  Lemma~\ref{lem:M1} says that
its restriction to \(U\) is a scalar multiple of \(\id_U\).  The only scalar map
on \(U\) with a non-zero kernel is the zero map.  Hence \(\theta(e_i)=0\).  Since
\(i\) was arbitrary, \(\theta(T)=0\).
\end{proof}

\begin{theorem}\label{thm:Melikian}
Let \(\K\) be an algebraically closed field of characteristic \(5\), and let
\(\uln=(n_1,n_2)\in\mathbb Z_{>0}^2\).  Let \(M=M(2;\uln)\).  Then
\[
    \HomLie(M)=\K\id_M .
\]
\end{theorem}

\begin{proof}
Let \(\varphi\in\HomLie(M)\).

\smallskip\noindent\textbf{The negative components.}
By Lemma~\ref{lem:M1}, there exists \(\lambda\in\K\) such that
\(\varphi|_U=\lambda\id_U\).  Let \(\psi=\varphi-\lambda\id_M\).
Then \(\psi\in\HomLie(M)\) and \(\psi(U)=0\).  Lemma~\ref{lem:M2} gives
\(\psi(c)=0\) and \(\psi(T)=0\).  Thus \(\psi(M_-)=0\).

\smallskip\noindent\textbf{The component \(M_0\).}
Let \(x\in M_1\).  By
Lemma~\ref{lem:HomLie-module}, \(\psi\bullet x\in\HomLie(M)\).  For \(u\in U\),
\[
    (\psi\bullet x)(u)
    =
    [\psi(u),x]-\psi([u,x])
    =
    -\psi([u,x]).
\]
Since \([U,x]\subseteq\K c\) and \(\psi(c)=0\), we have
\((\psi\bullet x)(U)=0\).
Applying Lemma~\ref{lem:M2} to \(\psi\bullet x\), we get
\((\psi\bullet x)(T)=0\).
For \(e\in T\), this means
\[
    0=(\psi\bullet x)(e)
    =
    [\psi(e),x]-\psi([e,x])
    =
    -\psi([e,x]).
\]
But \([T,M_1]=M_0\) by \eqref{eq:M-low6}.  Hence
\(\psi(M_0)=0\).

The relations \eqref{eq:M-low7}--\eqref{eq:M-low9} give
\([M_-,M_0]=M_-\), while Lemma~\ref{lem:M-cent2} gives
\(\Cent_M(M_-)=U\).  Moreover,
\([M_0,M_0]\simeq\mathfrak{sl}_2\), and \eqref{eq:M-low7} shows that this
algebra has no non-zero fixed vector in \(U\).  Hence
\(\Cent_M(M_-)\cap\Cent_M([M_0,M_0])=0\).  Lemma~\ref{lem:propagation},
applied with \(V=M_-\) and \(A=M_0\), gives \(\psi=0\).  Therefore
\(\varphi=\lambda\id_M\).
\end{proof}

\section{Filtered Cartan-type algebras and the classification}
\label{sec:filtered-cartan}

We now pass from graded to filtered Cartan type.  The compatibility property
for the associated graded algebra is
\cite[Property~(6.2.2)]{StradeI}; the assertions for filtered algebras of types
\(W\) and \(K\) follow from
\cite[Corollaries~6.1.7 and~6.6.3]{StradeI}.  We use the formulation in
\cite[Section~3.4.2]{Bois2009}; see also
\cite{BlockWilson1988,PremetStrade2007,PremetStrade2008,StradeIII}.

\begin{theorem}\label{thm:compatibility}
Let \(\K\) be an algebraically closed field of characteristic \(p>3\), and let
\(L\) be a finite-dimensional simple Lie algebra of filtered Cartan type.
Then \(L\) has a finite descending Lie filtration whose associated graded
algebra \(G=\gr L\) can be identified with a graded subalgebra of a full
Cartan-type algebra, and one of the following holds:
\[
\begin{gathered}
 G=W(m;\uln),\qquad
    m\geqslant1,\quad \uln\in\mathbb Z_{>0}^{m};\\
 S(m;\uln)^{(1)}\subseteq G\subseteq S(m;\uln),\qquad
    m\geqslant3,\quad \uln\in\mathbb Z_{>0}^{m};\\
 H(2r;\uln)^{(2)}\subseteq G\subseteq H(2r;\uln),\qquad
    r\geqslant1,\quad \uln\in\mathbb Z_{>0}^{2r};\\
 K(2r+1;\uln)^{(1)}\subseteq G\subseteq K(2r+1;\uln),\qquad
    r\geqslant1,\quad \uln\in\mathbb Z_{>0}^{2r+1}.
\end{gathered}
\]
All inclusions preserve the displayed gradings.  If the first case holds with
\(m=1\), then \(L\simeq W(1;n')\) for some \(n'\geqslant1\).
\end{theorem}

Suppose that \(L\) is not a Zassenhaus algebra.  In the first case of
Theorem~\ref{thm:compatibility}, the final assertion excludes \(m=1\), so
Theorem~\ref{thm:W} applies.  Theorems~\ref{thm:S}, \ref{thm:H}, and
\ref{thm:K} apply in the other three cases.  Thus
\(\HomLie(\gr L)=\K\id_{\gr L}\), and Lemma~\ref{lem:filtered} gives the
following result.

\begin{proposition}\label{prop:filtered-cartan}
Let \(\K\) be an algebraically closed field of characteristic \(p>3\), and let
\(L\) be a finite-dimensional simple Lie algebra of filtered Cartan type.  If
\(L\) is not isomorphic to a Zassenhaus algebra, then
\[
 \HomLie(L)=\K\id_L.
\]
\end{proposition}

\begin{proof}[Proof of Theorem~\ref{thm:main}]
In characteristic zero, every
finite-dimensional simple Lie algebra is of classical type, so
Proposition~\ref{prop:sl2} and Theorem~\ref{thm:classical} give the first and
third cases of Theorem~\ref{thm:main}.  Suppose that
\(\operatorname{char}\K=p>3\).  By the
Block--Wilson--Premet--Strade classification
\cite{BlockWilson1988,PremetStrade2007,PremetStrade2008}, every
finite-dimensional simple Lie algebra is of classical type, of filtered
Cartan type, or, when \(p=5\), a Melikian algebra.  Proposition~\ref{prop:sl2}, Theorem~\ref{thm:classical},
Propositions~\ref{prop:zassenhaus} and \ref{prop:filtered-cartan}, and
 Theorem~\ref{thm:Melikian} give the corresponding Hom--Lie spaces.  Combining
these results gives precisely the three cases in Theorem~\ref{thm:main}.
\end{proof}

The classification in Theorem~\ref{thm:main} also yields a direct consequence over arbitrary ground fields of characteristic different from \(2\) and \(3\). To apply the theorem after extending scalars to an algebraic closure, we consider absolutely simple Lie algebras, that is, Lie algebras that remain simple under this extension.For a finite-dimensional Lie algebra, the space of Hom-Lie structures commutes with extension of scalars, since the Hom-Jacobi identity is a homogeneous linear condition on the endomorphism. Consequently, both the scalar conslusion and the dimensions of the exceptional Hom-Lie spaces carry over to the original field, giving the following corollary.

%The following consequence over arbirary fields follows by extension of scalars. 
%We now return to arbitrary ground fields of characteristic different from \(2\) and \(3\), including characteristic zero. Let \(\overline \K\) be an algebraic closure of \(\K\). A finite-dimensional Lie algebra \(L\) over \(\K\) is called \emph{absolutely simple} if \(L_{\overline \K}=L\otimes_{\K}\overline \K\) is simple. In chararcteristic \(p>3\), we call \(L\) a \(\K\)-form of a Zassenhaus algebra if
%\[ L_{\overline \K}\cong W_{\overline \K}(1;n) \] for some \(n \geq 1\).

\begin{corollary}\label{cor:arbi-fields}
Let \(\K\) be an arbitrary field with \(\operatorname{char}\K \neq 2,3\), and let \(L\) be a finite-dimensional absolutely simple Lie algebra over \(\K\); that is \(L_{\overline \K}=L\otimes_{\K}\overline \K\) is simple, where \(\overline \K\) be an algebraic closure of \(\K\). Then
\[
\operatorname{HomLie}_\K(L) = \K\operatorname{id}_L
\]
unless \(\dim_\K L=3\), or \(\operatorname{char}\K=p>3\) and \(L_{\overline \K}\cong W_{\overline \K}(1;n)\) for some \(n \geq 1\).
In these two exceptional cases, the dimensions of \(\operatorname{HomLie}_\K(L)\) are \(6\) and \(p^n\), respectively.
\end{corollary}

\begin{proof}
For finite-dimensional \(L\), the Hom-Jacobi identity is a homogeneous linear system in the matrix entries of an endomorphism. Extension of scalars therefore gives 
\[
\operatorname{HomLie}_\K(L)\otimes_\K\overline \K
\cong
\operatorname{HomLie}_{\overline \K}(L_{\overline \K})
\]
see \cite[Section~4]{Zusmanovich2023}. Since \(L_{\overline \K}\) is simple, Theorem~\ref{thm:main} applies. The asserted alternatives and dimensions follow from the invariance of dimension under extension of scalars.
\end{proof}

\section*{Acknowledgements}

This work was partially supported by the National Natural Science Foundation
of China (grant no.~12101111).  H.~Zhu acknowledges support from an NTU
Research Scholarship.

\bigskip
\noindent
J.S.:\\
School of Mathematics, Southeast University, Nanjing, Jiangsu 210096, China\\
sjcsrc0927@163.com

\vspace{5mm}

\noindent
C.Z.:\\
College of Sciences, Northeastern University, Shenyang, Liaoning 110004, China\\
zhangchi@mail.neu.edu.cn

\vspace{5mm}

\noindent
H.Z.:\\
Division of Mathematical Sciences, Nanyang Technological University, 637371,
Singapore\\
zhuh0031@e.ntu.edu.sg


\begin{thebibliography}{99}

\bibitem{BlockWilson1988}
R.E.~Block and R.L.~Wilson,
\textit{Classification of the restricted simple Lie algebras},
J.~Algebra \textbf{114} (1988), no.~1, 115--259.

\bibitem{Bois2009}
J.M.~Bois,
\textit{Generators of simple Lie algebras in arbitrary characteristics},
Math.~Z. \textbf{262} (2009), no.~4, 715--741.

\bibitem{HLS}
J.T.~Hartwig, D.~Larsson, and S.D.~Silvestrov,
\textit{Deformations of Lie algebras using \(\sigma\)-derivations},
J.~Algebra \textbf{295} (2006), no.~2, 314--361.

\bibitem{JL}
Q.~Jin and X.~Li,
\textit{Hom-Lie algebra structures on semi-simple Lie algebras},
J.~Algebra \textbf{319} (2008), no.~4, 1398--1408.

\bibitem{LegerLuks2000}
G.F.~Leger and E.M.~Luks,
\textit{Generalized derivations of Lie algebras},
J.~Algebra \textbf{228} (2000), no.~1, 165--203.

\bibitem{MS}
A.~Makhlouf and S.D.~Silvestrov,
\textit{Hom-algebra structures},
J.~Gen.\ Lie Theory Appl.\ \textbf{2} (2008), no.~2, 51--64.

\bibitem{MakhloufZusmanovich2018}
A.~Makhlouf and P.~Zusmanovich,
\textit{Hom--Lie structures on Kac--Moody algebras},
J.~Algebra \textbf{515} (2018), 278--297.

\bibitem{PremetStrade2006}
A.~Premet and H.~Strade,
\textit{Classification of finite-dimensional simple Lie algebras in prime
characteristics},
in \textit{Representations of Algebraic Groups, Quantum Groups, and Lie Algebras},
Contemp.~Math. \textbf{413}, American Mathematical Society, Providence, RI,
2006, 185--214.

\bibitem{PremetStrade2007}
A.~Premet and H.~Strade,
\textit{Simple Lie algebras of small characteristic. V. The non-Melikian case},
J.~Algebra \textbf{314} (2007), no.~2, 664--692.

\bibitem{PremetStrade2008}
A.~Premet and H.~Strade,
\textit{Simple Lie algebras of small characteristic. VI. Completion of the
classification},
J.~Algebra \textbf{320} (2008), no.~9, 3559--3604.

\bibitem{StradeI}
H.~Strade,
\textit{Simple Lie Algebras over Fields of Positive Characteristic.
I. Structure Theory},
Walter de Gruyter, Berlin, 2004.

\bibitem{StradeIII}
H.~Strade,
\textit{Simple Lie Algebras over Fields of Positive Characteristic.
III. Completion of the Classification},
Walter de Gruyter, Berlin, 2013.

\bibitem{WuLiu2015}
M.~Wu and W.~Liu,
\textit{Multiplicative Hom-structures on finite-dimensional Lie algebras of
Cartan type},
Adv.~Math.~(China) \textbf{44} (2015), no.~3, 354--362.

\bibitem{XJL}
W.~Xie, Q.~Jin, and W.~Liu,
\textit{Hom-structures on semi-simple Lie algebras},
Open Math.\ \textbf{13} (2015), 617--630.

\bibitem{XL}
W.~Xie and W.~Liu,
\textit{Hom-structures on simple graded Lie algebras of finite growth},
J.~Algebra Appl.\ \textbf{16} (2017), no.~8, 1750154, 18~pp.

\bibitem{Zusmanovich2023}
P.~Zusmanovich,
\textit{When Hom--Lie structures form a Jordan algebra},
J.~Algebra Appl.\ \textbf{22} (2023), no.~9, 2350197, 12~pp.

\end{thebibliography}
\end{document}